\documentclass[a4paper]{paper}
\usepackage{amsmath,amsthm,amsfonts,mathrsfs}
\usepackage{amssymb}
\usepackage{tikz-cd}
\usepackage{xspace}
\usepackage{graphicx,booktabs}
\usepackage{paralist}
\usepackage{subfig}
\usepackage{hyperref}
\hypersetup{%
    pdfmenubar=true,       
    pdffitwindow=false,     
    pdfstartview={FitH},    
    pdftitle={Image Reconstruction for Multispectral CT},
    colorlinks=true,       
    linkcolor=red,          
    citecolor=red,        
    filecolor=magenta,      
    urlcolor=cyan,           
}
\usepackage{acro}
\usepackage{cleveref}
\usepackage{algorithm,algorithmic}
\makeatletter
\def\BState{\State\hskip-\ALG@thistlm}
\makeatother
\Crefname{equation}{}{}
\usepackage{todonotes}
\usepackage{inputenc}
\usepackage{dsfont}
\usepackage{amsbsy}
\usepackage{enumerate}
\usepackage{mathtools}

\newtheorem{definition}{Definition}[section]
\newtheorem{theorem}{Theorem}[section]
\newtheorem{corollary}[theorem]{Corollary}
\newtheorem{proposition}[theorem]{Proposition}
\newtheorem{lemma}[theorem]{Lemma}

\newtheorem{remark}[theorem]{Remark}

\numberwithin{equation}{section}

\title{A mathematical study of joint image reconstruction and motion estimation using optimal control}
\author{
 Zhentong Wei \thanks{LSEC, ICMSEC, Academy of Mathematics and Systems Science, Chinese Academy of Sciences, Beijing 100190, China. University of Chinese Academy of Sciences, Beijing 100190, China.}
 \and
 Chong Chen \thanks{LSEC, ICMSEC, Academy of Mathematics and Systems Science, Chinese Academy of Sciences, Beijing 100190, China. University of Chinese Academy of Sciences, Beijing 100190, China.}
}
\date{}

\begin{document}
\begin{sloppypar}

\maketitle

\begin{abstract}
Spatiotemporal dynamic medical imaging is critical in clinical applications, such as tomographic imaging of the heart or lung. To address such kind of spatiotemporal imaging problems, essentially, a time-dependent dynamic inverse problem, the variational model with intensity, edge feature and topology preservations was proposed for joint image reconstruction and motion estimation in the previous paper [C.~Chen, B.~Gris, and O.~{\"O}ktem, {\it SIAM J. Imaging Sci.}, 12 (2019), pp.~1686--1719], which is suitable to invert the time-dependent sparse sampling data for the motion target with large diffeomorphic deformations. However, the existence of solution to the model has not been given yet. In order to preserve its topological structure and edge feature of the motion target, the unknown velocity field in the model is restricted into the admissible Hilbert space, and the unknown template image is modeled in the space of bounded variation functions. Under this framework, this paper analyzes and proves the solution existence of its time-discretized version from the point view of optimal control. Specifically, there exists a constraint of transport equation in the equivalent optimal control model. We rigorously demonstrate the closure of the equation, including the solution existence and uniqueness, the stability of the associated nonlinear solution operator, and the convergence. Finally, the solution existence of that model can be concluded.
\end{abstract}

\begin{keywords}
  Spatiotemporal dynamic imaging, joint image reconstruction and motion estimation, variational model, solution existence, optimal control, transport equation
\end{keywords}


\section{Introduction}
Spatiotemporal dynamic medical imaging is critical in clinical applications, such as tomographic imaging of the heart or lung \cite{wang1999cardiac,schwarz2000implications,haldar2010spatiotemporal,chen2019new}. As an example, in PET/CT cardiac imaging, when data is acquired over a relatively long period of time (often in the range of minutes), the unknown motion of the organs leads to severe degradation in image quality \cite{rahmim2013four,gigengack2015motion}. Furthermore, low-dose or sparse sampling is often required to reduce the radiation or conduct fast scanning. Hence, sparse image reconstruction with high spatial and temporal resolutions becomes particularly important and very challenging. Joint image reconstruction and motion estimation plays an important role in spatiotemporal dynamic imaging, which embeds the motion estimation into the reconstruction process \cite{burger2018variational,chen2019new,huang2019dynamic, brehm2012self, tang2012fully}. 
This is a hard problem since both of the image sequence and velocity field are unknown, which can be attributed to an emerging time-dependent dynamic inverse problem \cite{kaltenbacher2021timedependent}. 

In the last few decades, several mathematical models for image reconstruction and motion estimation have been proposed and widely developed till now. A variational model based on gradient regularization for optical flow estimation is proposed in \cite{horn1981determining} (see also \cite{schnorr1991determining}). The authors in \cite{aubert1999computing} used $\mathscr{L}^{1}$ norm for the optical flow constraint and considered the edge-preserving properties and homogeneous regions (see also \cite{aubert1999mathematical,kornprobst1998contribution}). Furthermore, an optimal control formulation for determining optical flow is considered in \cite{borzi2003optimal}. The authors in \cite{chen2011image} modified the model in \cite{borzi2003optimal}, and analyzed the well-posedness of the corresponding minimization problem under the assumption of vanishing divergence of the velocity field. To estimate the motion, a displacement field was treated as the motion filed in \cite{blume2010joint}, and the authors proposed a joint image estimate and deformable motion model by regularizing the motion field in positron emission tomography (PET). Furthermore, a B-spline-based discretized scheme for the model in \cite{blume2010joint} and a parameter-free motion regularization approach were presented in \cite{blume2012joint}. In the presence of severe artifacts, the authors in \cite{brehm2013artifact} considered how to estimate the respiratory motion, and further proposed a motion-compensated image reconstruction model by applying the estimated motion fields. Due to the necessity of double gating, the simultaneous respiratory and cardiac gating, in vivo micro-CT imaging of small animals, a cardiorespiratory motion-compensated micro-CT image reconstruction model was proposed in \cite{brehm2015cardiorespiratory}. The authors in \cite{liu20155d} developed a 4D cone-beam computed tomography reconstruction method from the breathing motion model in \cite{low2005novel}. 
The authors in \cite{chen2018indirect} proposed an image reconstruction method called indirect image registration based on motion estimation with the given template image and indirect observations, which combines the two challenges of sparse image reconstruction and image registration (also see \cite{oektem2017shape,gris2020image}). Assuming the exact deformation of the moving object is known, some motion compensation strategies in tomography were studied in \cite{hahn2021motion}. To address the problem of spatiotemporal dynamic imaging, a joint total variation (TV)-TV model was proposed in \cite{burger2018variational}. Considering topology preservation of the reconstructed image sequence, the authors in \cite{chen2019new} proposed a joint image reconstruction and motion estimation model under the framework of large deformation diffeomorphic image mapping (LDDMM). The references, for instance, \cite{beg2005computing, bruveris2015geometry, trouve1998diffeomorphisms, dupuis1998variational, younes2010shapes}, presented the details of LDDMM. Moreover, a concept called diffeomorphic optimal transportation (DOT) was introduced in \cite{chen2021spatiotemporal}, and the author further proposed a new variational model using DOT for spatiotemporal imaging involving mass-preserving large diffeomorphic deformations. 

Although the optical flow is often used to describe the kinematic behavior of objects \cite{fran1990}, few of literatures considered its mathematical analysis. The key point to this problem is how to deal with the existence and uniqueness of solution to the transport equation, which is still an active field of research. In order for solution to satisfy the desired chain rule, the concept of renormalized solutions is proposed in \cite{diperna1989ordinary}, and the authors  proved the existence and uniqueness of solution to the transport equation in the space of $W^{1,1}$. Furthermore, the velocity field is extended to $BV$ regularity in \cite{ambrosio2004transport}. The author deduced that if the space of velocity fields is $BV$ with respect to the spatial variables, then any distributional solution is a renormalized solution, and further derived the existence and uniqueness of solution to the transport equation. Moreover, based on the works in \cite{ambrosio2004transport, crippa2008flow, crippa2014initial, crippa2014noteb}, the existence and uniqueness of solution to the transport equation was demonstrated under the velocity field with mild regularity assumptions in \cite{jarde2018analysis} (see also \cite{jarde2019existence}).

\textbf{Contributions.} The solution existence  of the time-discretized variational model with intensity, edge feature and topology preservations proposed in \cite{chen2019new} is proved from the point view of optimal control, where the unknown velocity field is restricted into the admissible Hilbert space, and the unknown template image (the initial value of the sequence image) is modeled in the space of bounded variation functions. Specifically, we rigorously demonstrate the closure of the constraint of transport equation in the equivalent optimal control model, including the solution existence and uniqueness, the stability of the associated nonlinear solution operator, and the convergence. 

Compared to the previous work, the restriction on the divergence of the velocity field no longer needs to be considered, the initial value of the sequence image (template image) is treated as an unknown, and the edge-preserving property of the image is also incorporated. These factors lead to a more difficult problem than before that we need to analyze. Furthermore, the obtained results are still valid for $N$-dimensional space, where $N \geq 2$, and also compensate the theory of the emerging  time-dependent dynamic inverse problem. 

\textbf{Outline.} In \cref{sec2}, we briefly introduce some preliminaries and the problem setting that are needed in the paper. \Cref{sec3} gives the detailed proofs of the existence and uniqueness of solution to the constraint of transport equation. Furthermore, the stability of the nonlinear solution operator of the associated constraint is proved, and the convergence of the constraint is also verified. Finally, the existence of solution to the considered minimization problem is concluded in \cref{sec4}. 


\section{Mathematical preliminaries}\label{sec2}

Here we introduce some required mathematical preliminaries. Let us first introduce the Hadamard inequality. 

\begin{lemma}[Hadamard inequality \cite{horn2012matrix}]\label{Hadamard}
  Let $A \in \mathbb{R}^{N \times N}$. Then
  \[
      |\operatorname{det} A| \leq \prod_{i=1}^{N}\left(\sum_{j=1}^{N}\left|A_{i j}\right|^{2}\right)^{\frac{1}{2}} \text { and }\ |\operatorname{det} A| \leq \prod_{j=1}^{N}\left(\sum_{i=1}^{N}\left|A_{i j}\right|^{2}\right)^{\frac{1}{2}}.
  \]
\end{lemma}

Next, we give a brief introduction for the space of $BV$ functions, flow of diffeomorphisms and  problem setting.

\subsection{Space of $BV$ functions}
First of all, we recall some basic facts about the space of $BV$ functions that plays an important role in image analysis \cite{rudin1992nonlinear} (see \cite{ambrosio2000functions,evans2015measure}).
\begin{definition}[$BV$ functions \cite{ambrosio2000functions}]
  Assume $\Omega \subset \mathbb{R}^{N}$ is an open set and  $f \in \mathscr{L}^{1}(\Omega)$. Say that $f$ is a function of bounded variation in $\Omega$ if the distributional derivative of $f$ is representable by a finite Radon measure in  $\Omega$, i.e., if for any $i=1 \ldots N$, it holds that
  \[
        \int_{\Omega} f \frac{\partial \varphi}{\partial x_{i}} d x=-\int_{\Omega} \varphi d D_{i} f, \quad \forall \varphi \in \mathscr{C}_{c}^{\infty}(\Omega) 
  \]
  for some $\mathbb{R}^{N}$-valued measure $Df = (D_{1}f, \dots, D_{N}f)$ in $\Omega$. The vector space of all functions of bounded variation in $\Omega$ is denoted by $BV(\Omega)$. Furthermore, define
  \[
    BV_{0}(\Omega) \triangleq \{f\in BV(\Omega) : Tr(f) = 0\},
  \]
  where $Tr$ denotes the trace operator.
\end{definition}

\begin{definition}[Total variation \cite{aubert2006mathematical}]
  Assume $\Omega \subset \mathbb{R}^{N}$ is an open set and $f \in \mathscr{L}^{1}(\Omega)$. The total variation of $f$ in $\Omega$ is defined by
  \begin{multline}
      |Df|(\Omega)\\
      \triangleq \sup \left\{\int_{\Omega} f \operatorname{div} \varphi d x :\varphi=\left(\varphi_{1}, \varphi_{2}, \ldots, \varphi_{N}\right) \in \mathscr{C}_{0}^{1}(\Omega;\mathbb{R}^{N})\ \text{and}\ \|\varphi\|_{L^{\infty}(\Omega)} \leq 1\right\},
  \end{multline} 
  where $\operatorname{div} \varphi=\sum_{i=1}^{N} \frac{\partial \varphi_{i}}{\partial x_{i}}(x)$ and $\|\varphi\|_{L^{\infty}(\Omega)}=\sup _{x} \sqrt{\sum_{i} \varphi_{i}^{2}(x)}$.
\end{definition}

We subsequently recall some helpful properties of $BV$ functions that are requisite. Before that, we give the definition of mollifier operator.
\begin{definition}[\cite{evans2022partial}]\label{define2.7}
    \begin{enumerate}[(i)]
        \item Define $\rho \in \mathscr{C}^{\infty}(\mathbb{R}^{N}, \mathbb{R})$ by
        \[
        \rho(x) \triangleq \left\{\begin{array}{ll} 
            C \exp\left(\frac{1}{|x|^{2} -1}\right) &\text { if }~ |x| \leq 1, \\[2mm]
            0 &\text { if }~ |x| \geq 1. \end{array}\right.
        \]
        The constant $C > 0$ is selected such that $\int_{\mathbb{R}^{N}}\rho dx = 1$.
        \item For $\varepsilon > 0$, set
        \[
        \rho_{\varepsilon}(x) \triangleq \frac{1}{\varepsilon^{N}}\rho\left(\frac{x}{\varepsilon}\right).
        \]
    \end{enumerate}
    We call $\rho$ the standard mollifier. The functions $\rho_{\varepsilon}$ are $\mathscr{C}^{\infty}$ and satisfy
    \[
    \int_{\mathbb{R}^{N}}\rho_{\varepsilon} dx = 1,\quad \operatorname{supp}(\rho_{\varepsilon}) \subset B(0,\varepsilon).
    \]
\end{definition}

\begin{proposition}[Property of $Df$ \cite{ambrosio2000functions}]\label{pro2.4}
  Assume $\Omega \subset \mathbb{R}^{N}$ is an open set and $f \in BV(\Omega)$, if $\rho$ is the convolution kernel in \cref{define2.7} and $\Omega_{\varepsilon} = \{ x\in \Omega : dist(x, \partial \Omega) > \varepsilon\}$, where $\varepsilon > 0$. Then, it holds
  \[
  \nabla\left(f * \rho_{\varepsilon}\right)=D f * \rho_{\varepsilon} \quad \text { in }~  \Omega_{\varepsilon}, 
  \]
  where $Df$ is the Radon measure of $f$.
\end{proposition}

\begin{theorem}[\cite{ambrosio2000functions}]\label{th2.5}
  Let $\Omega \subset \mathbb{R}^{N}$ be an open set, $\mu = (\mu_{1}, \dots, \mu_{m})$ be a Radon measure in $\Omega$ and let $(\rho_{\varepsilon})_{\varepsilon > 0}$ be a family of mollifiers. Then, the measures $\mu_{\varepsilon}=\mu * \rho_{\varepsilon} \mathcal{L}^{N}$ locally weakly$^{\star}$ converge in $\Omega$ to $\mu$ as $\varepsilon \rightarrow 0$ and the estimate
  \[
  \int_{E}\left|\mu * \rho_{\varepsilon}\right|(x) d x \leq|\mu|\left(\Omega_{\varepsilon}\right) 
  \]
  holds whenever $E \subset \Omega_{\varepsilon}$ is a Borel set. Here, $\mathcal{L}^{N}$ is Lebesgue outer measure in $\mathbb{R}^{N}$.
\end{theorem} 

The following results are two main tools used in this paper.
\begin{theorem}[Weak$^{\star}$ topology \cite{ambrosio2000functions}]\label{th2.6}
    Assume $\Omega \subset \mathbb{R}^{N}$ is an open set and $\{f_{k}\}_{k} \subset BV(\Omega)$. Then, $\{f_{k}\}_{k}$  converges weakly$^{\star}$ to $f$ in $BV(\Omega)$ if and only if $\{f_{k}\}_{k}$ is bounded in $BV(\Omega)$ and converges to $f$ in $\mathscr{L}^{1}(\Omega)$.
\end{theorem}

Equipped with this topology, $BV(\Omega)$ possesses compactness that is given in the following theorem.
\begin{theorem}[Compactness \cite{aubert2006mathematical}]\label{th2.7}
    Assume $\Omega \subset \mathbb{R}^{N}$ is bounded and has a Lipschitz boundary. Every uniformly bounded sequence $\{f_{k}\}_{k} \subset BV(\Omega)$ is relatively compact in $\mathscr{L}^{p}(\Omega)$ for $1 \leq p<\frac{N}{N-1}, N \geq 1$. Moreover, there exists a subsequence $\{f_{k'}\}_{k'}$ of  $\{f_{k}\}_{k}$ converging weakly$^{\star}$ to $f$ in $BV(\Omega)$. We also know that $BV(\Omega)$ is continuously embedded into $\mathscr{L}^{p}(\Omega)$ with $p = + \infty$ if $N = 1$, and $p = \frac{N}{N-1}$ otherwise.
\end{theorem}

\subsection{Flow of diffeomorphisms}

Let us review some results concerning the flow of diffeomorphisms (see, for instance, \cite{younes2010shapes,miller2015hamiltonian}). For brevity, we write $f_{t}(\cdot) \triangleq f(t,\cdot)$.
\begin{definition}[\cite{chen2011image,crippa2008flow}]
  Given a velocity field $\boldsymbol{v} : [0, 1] \times \Omega \rightarrow \mathbb{R}^{N}$, the classical flow of velocity field $\boldsymbol{v}$ is a map 
  \[
  \phi_{t}(x) :[0,1] \times \Omega \longrightarrow \Omega,
  \]
  which satisfies the following ODE:
  \begin{equation}\label{2.4}
      \left\{\begin{array}{ll}\partial_{t} \phi_{t}(x)=\boldsymbol{v}(t, \phi_{t}( x)) &\text { in }~ ] 0, 1] \times \Omega, \\[2mm] \phi_{0}(x)=x &\text { in }~ \Omega.\end{array}\right.
  \end{equation}   
\end{definition}

Below, the definition is very important for the existence of the flow of diffeomorphisms.
\begin{definition}[Admissible space \cite{vialard2009hamiltonian, younes2010shapes}]\label{def2.2}
    A Hilbert space $\mathscr{V} \subset \mathscr{C}_{0}^{k}(\Omega, \mathbb{R}^{N})$ is k-admissible if it is (canonically) embedded in $\mathscr{C}_{0}^{k}(\Omega, \mathbb{R}^{N})$, i.e., there exists a positive constant C such that, for all $\boldsymbol{v} \in \mathscr{V}$, 
    \[
    ||\boldsymbol{v}||_{\mathscr{V}} \geqslant C ||\boldsymbol{v}||_{k,\infty}.
    \]
    Particularly, $\mathscr{V}$ is said admissible if $k=1$.
\end{definition}

Next, we denote the space of velocity fields as 
\[
    \mathscr{L}^{p}([0,1], \mathscr{V})\triangleq \left\{\boldsymbol{v}: \boldsymbol{v}(t, \cdot) \in \mathscr{V} \text { and }\|\boldsymbol{v}\|_{\mathscr{L}^{p}([0,1], \mathscr{V})}<\infty \text { for } p \in [1, \infty]\right\} 
\]
with the norm
\[
\|\boldsymbol{v}\|_{\mathscr{L}^{p}([0,1], \mathscr{V})}\triangleq \left(\int_{0}^{1}\|\boldsymbol{v}(t, \cdot)\|_{\mathscr{V}}^{p} \mathrm{~d} t\right)^{1 / p} \quad \text { for }~ p \in [1, \infty[,
\]
and 
\[
\|\boldsymbol{v}\|_{\mathscr{L}^{\infty}([0,1], \mathscr{V})}\triangleq \mathrm{ess} \underset{t\in [0,1 ]}{\sup} \|\boldsymbol{v}(t, \cdot)\|_{\mathscr{V}}  \quad \text { for } ~ p = \infty.
\]
To simplify the notation, let $\mathscr{L}_{\mathscr{V}}^{p}(\Omega)$ denote $\mathscr{L}^{p}([0,1], \mathscr{V})$.

The classical flow can be a flow of diffeomorphisms via an admissible velocity field as follow.
\begin{theorem}[\cite{younes2010shapes, bruveris2015geometry}]\label{th2.1}
    Let $\mathscr{V}$ be an admissible Hilbert space and $\boldsymbol{v} \in \mathscr{L}_{\mathscr{V}}^{2}(\Omega)$ be a velocity field. Then, the ODE in \eqref{2.4} admits a unique solution $\phi^{\boldsymbol{v}} \in \mathscr{C}^{1}([0,1] \times \Omega, \Omega)$, such that for $t \in [0,1]$, the mapping $\phi^{\boldsymbol{v}}_{t}:\Omega \rightarrow \Omega$ is a $\mathscr{C}^{1}$-diffeomorphism on $\Omega$.
\end{theorem}

\begin{definition}[\cite{younes2010shapes}]
   Let $\boldsymbol{v} \in \mathscr{L}_{\mathscr{V}}^{2}(\Omega)$. We denote by $\phi^{\boldsymbol{v}}_{s,t}(x)$ the solution at time $t$ of the ODE in \eqref{2.4} with initial
   condition $\phi_s(x) = x$. The function $(t, x) \mapsto \phi^{\boldsymbol{v}}_{s,t}(x)$ is called the flow associated to $\boldsymbol{v}$ starting at $s$. It is defined on $[0, 1]\times \Omega$ and takes values in $\Omega$.
\end{definition}
\begin{proposition}[\cite{younes2010shapes}]\label{pro2.13}
    If $\boldsymbol{v} \in \mathscr{L}_{\mathscr{V}}^{2}(\Omega)$ and $s, r, t \in [0, 1]$, then
    \[
        \phi^{\boldsymbol{v}}_{s,t} = \phi^{\boldsymbol{v}}_{r,t} \circ \phi^{\boldsymbol{v}}_{s,r}.
    \]
    In particularly, $\phi^{\boldsymbol{v}}_{s,t} \circ \phi^{\boldsymbol{v}}_{t,s} = \mathrm{Id}$ and $\phi^{\boldsymbol{v}}_{s,t}$ is invertible for all $s$ and $t$.
\end{proposition}
From \cref{pro2.13}, we immediately get
\begin{equation}\label{2.7}
    \phi^{\boldsymbol{v}}_{t} = \phi^{\boldsymbol{v}}_{0,t},\quad (\phi^{\boldsymbol{v}}_{t})^{-1} = \phi^{\boldsymbol{v}}_{t,0}.
\end{equation}
    


\subsection{Problem setting}    
The problem that will be studied is whether there exists a minimizer of the general variational model for the spatiotemporal inverse problem in \cite{chen2019new}, which is an optimal control problem  defined by
\begin{equation}\label{2.1}
\begin{aligned}
    &\min _{\substack{I \in \mathscr{X} \\ \phi_{t_{i}}\in \mathscr{G}}} \mathcal{J}\left(I, \phi_{t_{i}} \right) \triangleq \frac{1}{T}\sum_{i = 1}^{T}\big(\mathcal{D}\left(\mathcal{T}_{t_{i}}\left(\phi_{t_{i}}\# I\right), g_{t_{i}}\right)+\mu_{2} \mathcal{R}_{2}\left(\phi_{t_{i}}\right) \big)+\mu_{1} \mathcal{R}_{1}(I) \\[2mm]
    & \quad \text { s.t. } \mathcal{E}\left(I, \phi_{t_{i}}\right) = 0, 
\end{aligned}   
\end{equation}
where $\mathscr{X}$ is the reconstruction space, i.e., the space of all possible images on a fixed domain $\Omega$, $\mathscr{Y}$ is the data space, i.e., the space of all possible data on a fixed domain $\widetilde{\Omega}$ and $g_{t_{i}} \in \mathscr{Y}$ is the measure data. Moreover, $\mathscr{G}$ is the group of diffeomorphisms on  $\Omega$ and $\phi_{t_{i}} \in \mathscr{G}:\Omega \rightarrow \Omega$ is continuously differentiable, where $\phi^{-1}_{t_{i}}$ is the inverse of $\phi_{t_{i}}$ \cite{beg2005computing, tu2011manifolds, chen2019new}. Furthermore, $\mathcal{T}_{t_{i}} : \mathscr{X} \rightarrow \mathscr{Y}$ is a time-dependent forward operator (e.g., the Radon transform with different geometric parameters in CT scanning) \cite{chen2019new, natterer2001mathematics}. The above $\mathcal{D} : \mathscr{Y} \times \mathscr{Y} \rightarrow \mathbb{R}_{+}$ is the data discrepancy function, $\mathcal{E} : \mathscr{X} \times \mathscr{G} \rightarrow \mathbb{R}^{l}$ is the constraint, $\mathcal{R}_{1} : \mathscr{X} \rightarrow \mathbb{R}_{+}$ and $\mathcal{R}_{2} : \mathscr{G} \rightarrow \mathbb{R}_{+}$ are the regularization functionals.

For deformation $\phi_{t_{i}} \in \mathscr{G}$ and image $I : \Omega \rightarrow \mathbb{R}$, $\phi_{t_{i}}\#I$ that defines a group action of $\mathscr{G}$ on the set of all images models how the deformation $\phi_{t_{i}}$ acts on the image $I$ \cite{grenander2006pattern, miller2015hamiltonian, thomson1917growth, trouve2015shape, younes2010shapes}. Here the following intensity-preserving deformation is considered 
\begin{equation}\label{2.2}
    \phi_{t_{i}}\#I = I \circ \phi^{-1}_{t_{i}},
\end{equation}
where $\circ$ denote the function composition. For its Lagrangian description, we refer e.g. to \cite{temam2005mathematical}. 

Next, given the velocity field $\boldsymbol{v} \in \mathscr{L}_{\mathscr{V}}^{2}(\Omega)$, the flow $\phi^{\boldsymbol{v}}_{t_{i}}$ in \eqref{2.1} generated by the flow equation \eqref{2.4} is a diffeomorphism on $\Omega$ via \cref{th2.1}. The authors in $\cite{chen2019new}$ proposed the regularization function $\mathcal{R}_{2}$ defined by
\[            
    \mathcal{R}_{2}\left(\phi^{\boldsymbol{v}}_{t_{i}}\right)\triangleq \int_{0}^{t_{i}}\|\boldsymbol{v}(\tau, \cdot)\|_{\mathscr{V}}^{2} \mathrm{~d} \tau.
\]
Therefore, problem \eqref{2.1} becomes
\begin{equation}\label{2.10}
  \begin{aligned}
      &\min _{\substack{I \in \mathscr{X} \\ \boldsymbol{v} \in \mathscr{L}_{\mathscr{V}}^{2}}} \mathcal{J}\left(I, \boldsymbol{v} \right) \triangleq \frac{1}{T}\sum_{i = 1}^{T}\big(\mathcal{D}\left(\mathcal{T}_{t_{i}}\left(\phi^{\boldsymbol{v}}_{0,t_{i}} \# I\right), g_{t_{i}}\right)+\mu_{2} \int_{0}^{t_{i}}\|\boldsymbol{v}(\tau, \cdot)\|_{\mathscr{V}}^{2} \mathrm{~d} \tau \big)+\mu_{1} \mathcal{R}_{1}(I) \\[2mm]
  & \quad \text { s.t. } \phi^{\boldsymbol{v}}_{0,t_{i}}\  \text{solves ODE}\  \eqref{2.4}. 
  \end{aligned}
\end{equation}
The following theorem shows that problem \eqref{2.10} is equivalent to a PDE-constrained optimal control problem.

\begin{theorem}[\cite{chen2019new}]\label{the2.7}
  Let $\mathscr{X}$ be a space of real-valued functions that are sufficiently smooth (e.g., the space of differentiable functions in distribution theory). Let $f_{0} \in \mathscr{X}$ and $f : [0,1] \times \Omega \rightarrow \mathbb{R}$ be defined as
  \[
      f({t_{i}}, \cdot) \triangleq \phi^{\boldsymbol{v}}_{0,t_{i}}\# I \quad \text { for }~  t_{i} \in [0, 1],
  \]
  where $\phi^{\boldsymbol{v}}_{0,t_{i}}$ is a diffeomorphism on $\Omega$ given by \eqref{2.7}. Assume furthermore that $f_{t_{i}} \in \mathscr{X}$ for $t_{i} \in [0, 1]$. Then, \eqref{2.10} with the group action given by the intensity-preserving deformation in \eqref{2.2} is equivalent to 
  \begin{equation}\label{transport equation}
      \begin{aligned}
          & \min _{\substack{f_{0} \in \mathscr{X} \\
          \boldsymbol{v} \in \mathscr{L}_{\mathscr{V}}^2(\Omega)}} \mathcal{J}(f_{0},\boldsymbol{v}) \triangleq \frac{1}{T}\sum_{i=1}^{T}\big(\mathcal{D}\left(\mathcal{T}_{t_{i}}(f_{t_{i}}), g_{t_{i}}\right)+\mu_2 \int_0^{t_{i}}\|\boldsymbol{v}(\tau, \cdot)\|_{\mathscr{V}}^2 \mathrm{~d} \tau\big)+\mu_1 \mathcal{R}_1(f_{0}) \\[2mm]
          &\quad \text { s.t. } \left\{\begin{array}{ll} \partial_t f(t, \cdot)+\langle\nabla f(t, \cdot), \boldsymbol{v}(t, \cdot)\rangle_{\mathbb{R}^N}=0 \quad \text{in}~ ]0,1] \times \Omega,\\[2mm]
          f(0, \cdot)=f_{0} \qquad \text { in }~\Omega.\end{array}\right.
      \end{aligned}
  \end{equation}
\end{theorem} 
Hence, problem \eqref{2.1} can be solved from the perspective of PDE-constrained optimal control as in $\cite{hinze2008optimization, lions1971optimal}$. 

\section{Closure of the PDE constraint}\label{sec3} 
This section proves the closure of the PDE constraint in \eqref{transport equation} when $f_{0} \in  \mathscr{X} = BV(\Omega) \cap \mathscr{L}^{\infty}(\Omega)$ and $\boldsymbol{v} \in \mathscr{L}_{\mathscr{V}}^{2}(\Omega)$. Subsequently, the $\Omega\subset \mathbb{R}^{N}$ with $N \geq 2$ is assumed to be open, bounded and have a Lipschitz boundary. 

We begin by establishing the existence and uniqueness of the following PDE constraint, namely, the transport equation
\begin{equation}\label{transport}
    \left\{\begin{array}{ll} \partial_t f(t, \cdot)+\langle\nabla f(t, \cdot), \boldsymbol{v}(t, \cdot)\rangle_{\mathbb{R}^N}=0 \quad \text{in}~ ]0,1] \times \Omega,\\[2mm]
            f(0, \cdot)=f_{0} \qquad \text { in }~ \Omega. \end{array}\right.
\end{equation}
\subsection{Existence and uniqueness of the solution}
Let us first define the nonlinear solution operator of transport equation \eqref{transport} as follow.
\[
    \begin{array}{c}
        S_{tran}: \mathscr{X} \times \mathscr{L}_{\mathscr{V}}^{2}(\Omega) \longrightarrow Z \\[2mm]
        \quad \left(f_{0}, \boldsymbol{v}\right) \longmapsto f.
    \end{array}
\]

In the light of \cite{chen2011image}, there exists no classic solution to transport equation \eqref{transport} equipped with a non-differentiable initial value. Hence, we define the weak solution to transport equation \eqref{transport} as below.
\begin{definition}[Weak solution]\label{def3.1}
    If $\boldsymbol{v}$ and $f_{0}$ are locally summable functions such that the distributional divergence of $\boldsymbol{v}$ is locally summable, then we say that a function $f : [0,1] \times \Omega \rightarrow \mathbb{R}$ is a weak solution of \eqref{transport} if the following identity holds for $\varphi \in \mathscr{C}_{c}^{\infty}([0,1[\times\Omega) :$
    \begin{equation}\label{weak solution of transport}
        \int_{0}^{1} \int_{\Omega} f\left[\partial_{t} \varphi+\varphi \operatorname{div} \boldsymbol{v}+\boldsymbol{v} \cdot \nabla \varphi\right] d x d t=-\int_{\Omega} f_{0}(x) \varphi(0, x) d x.
    \end{equation}
\end{definition}

Subsequently, the existence of the weak solution of transport equation \eqref{transport} will be proved. Before that, we recall some results that were established in \cite{chen2011image} and \cite{crippa2008flow}.
\begin{theorem}[\cite{chen2011image,crippa2008flow}]\label{th3.1}
    Let $f_{0} \in \mathscr{C}^{1}(\Omega)$ and $\phi^{\boldsymbol{v}}_{t}$ is the associated flow of velocity field $\boldsymbol{v} \in \mathscr{L}_{\mathscr{V}}^{2}(\Omega)$ in $\Omega$ for $t\in [0,1]$. Then, transport equation \eqref{transport} has the unique solution
    \[
        f(t,x) = f_{0} \circ (\phi^{\boldsymbol{v}}_{t})^{-1}(x).
    \]
\end{theorem}

\begin{theorem}[\cite{chen2011image}]\label{th3.2}
  Let $f_{0} \in BV(\Omega)$ and  $\rho_{\varepsilon}$ be the mollifier in  \cref{define2.7}. Furthermore, $\phi$ and $\phi^{-1}$ are diffeomorphisms and Lipschitz continuous in $\Omega$. Then, the sequence $\{(f_{0} \ast \rho_{\varepsilon})\circ \phi\}_{\varepsilon}$ converges weakly$^{\star}$ to $f_{0} \circ \phi$ in $BV(\Omega)$ as $\varepsilon \rightarrow 0$.
\end{theorem}
 
\begin{theorem}[\cite{chen2011image}]\label{th3.3}
  Let $f_{0} \in BV(\Omega)$ and $\rho_{\varepsilon}$ be the mollifier in \cref{define2.7}. Furthermore, $\phi_{t}$ and $\phi^{-1}_{t}$ are diffeomorphisms in $\Omega$ for $t\in [0,1]$. Moreover, $\phi(\cdot,x)$ is absolutely continuous in $[0,1]$ for $x \in \Omega$. Define
  \[
      f_{\varepsilon}(t,x) \triangleq (f_{0} \ast \rho_{\varepsilon})\circ \phi(t,x).
  \]
  Then, $f_{\varepsilon} \in \mathscr{C}([0,1];BV(\Omega))$.
\end{theorem}

Now, we can prove the existence of the weak solution of transport equation \eqref{transport} without the condition $\operatorname{div}\boldsymbol{v} = 0$. 
\begin{theorem}[Existence of weak solution]\label{transport solution}
 Let $\boldsymbol{v} \in \mathscr{L}_{\mathscr{V}}^{2}(\Omega)$ with $\mathrm{div} \boldsymbol{v} \in \mathscr{L}^{N}([0,1] \times \Omega)$ and  $f_{0} \in BV(\Omega)$. Furthermore, $\phi^{\boldsymbol{v}}_{t}$ is the associated flow of velocity field $\boldsymbol{v} \in \mathscr{L}_{\mathscr{V}}^{2}(\Omega)$ in $\Omega$ for $t\in [0,1]$. Define
  \[
      f(t,x) \triangleq f_{0} \circ (\phi^{\boldsymbol{v}}_{t})^{-1}(x).
  \]
  Then, $f \in \mathscr{L}^{\infty}([0,1];BV(\Omega))$ and is a weak solution to  transport equation \eqref{transport}.
\end{theorem}
\begin{proof}
  Let $\rho_{\varepsilon}$ be the mollifier in \cref{define2.7}. Then, we consider a modification of transport equation \eqref{transport} 
  \begin{equation}\label{regularizde transport}
          \left\{\begin{array}{ll} \partial_t f(t, \cdot)+\langle\nabla f(t, \cdot), \boldsymbol{v}(t, \cdot)\rangle_{\mathbb{R}^N}=0 \quad \text{in}~ ]0,1] \times \Omega,\\[2mm]
          f(0, \cdot)=f_{0} \ast \rho_{\varepsilon} \qquad \text { in }~ \Omega.\end{array}\right.
  \end{equation}
  By \cref{th3.1}, the solution to \cref{regularizde transport} exists and is unique, can be written as 
  \[
      f_{\varepsilon}(t,x) = (f_{0} \ast \rho_{\varepsilon}) \circ (\phi^{\boldsymbol{v}}_{t})^{-1}(x).
  \]
  According to \cref{th3.2}, $\{(f_{0} \ast \rho_{\varepsilon}) \circ (\phi^{\boldsymbol{v}}_{t})^{-1}\}_{\varepsilon}$ converges weakly$^{\star}$ to $f_{0} \circ (\phi^{\boldsymbol{v}}_{t})^{-1}$ in $BV(\Omega)$ as $\varepsilon \rightarrow 0$. Let $f(t,x) = f_{0} \circ (\phi^{\boldsymbol{v}}_{t})^{-1}(x)$, then $f(t,\cdot) \in BV(\Omega)$ for $t \in [0,1]$.
  Moreover, \cref{th2.6} and \cref{th3.3} show that $\{f_{\varepsilon}\}_{\varepsilon}$ is uniformly bounded in $\mathscr{L}^{\infty}([0,1];BV(\Omega))$. 
  
  Given $t \in [0, 1]$, \cref{th2.7} implies that there exists a subsequence $\{f_{\varepsilon_{k}}(t,\cdot)\}_{\varepsilon_{k}}$ of $\{f_{\varepsilon}(t,\cdot)\}_{\varepsilon}$ converges weakly to $f(t,\cdot)$ in $\mathscr{L}^{\frac{N}{N-1}}(\Omega)$. Then, we derive 
  \[
      \int_{\Omega}f_{\varepsilon_{k}}(t,x)\varphi(x)dx \longrightarrow \int_{\Omega}f(t,x)\varphi(x)dx\quad \text{as}~ \varepsilon_{k} \longrightarrow 0 
  \]
  for $\varphi \in \mathscr{L}^{N}(\Omega)$.
  Setting $(\phi^{\boldsymbol{v}}_{t})^{-1}(x) = y$ and $\varphi(x) = 1$, we observe 
  \[
      \begin{aligned}
          \Big| \int_{\Omega}f_{\varepsilon_{k}}(t,x)dx \Big|
          &=\Big| \int_{\Omega}(f_{0} \ast \rho_{\varepsilon_{k}}) \circ (\phi^{\boldsymbol{v}}_{t})^{-1}(x)dx \Big|\\[2mm]
          &=\Big| \int_{\Omega}f_{0} \ast \rho_{\varepsilon_{k}}(y) |\operatorname{det}\nabla \phi^{\boldsymbol{v}}_{t}(y)|dy \Big|\\[2mm]
          &\leq\|\operatorname{det}\nabla\phi^{\boldsymbol{v}}_{t}\|_{\mathscr{L}^{\infty}(\Omega)}\int_{\Omega}\big|f_{0} \ast \rho_{\varepsilon_{k}}(y) \big| dy \\[2mm]
          &\leq\|\operatorname{det}\nabla\phi^{\boldsymbol{v}}_{t}\|_{\mathscr{L}^{\infty}(\Omega)}\|f_{0}\|_{\mathscr{L}^{1}(\Omega)}.
      \end{aligned}
  \]
  \Cref{Hadamard} implies that
  \[
      \|\det\nabla\phi^{\boldsymbol{v}}_{t}\|_{\mathscr{L}^{\infty}(\Omega)} \leq N^{\frac{N}{2}} Lip(\phi^{\boldsymbol{v}}_{t})^{N},
  \]
  where $Lip(\phi^{\boldsymbol{v}}_{t})$ denotes the Lipschitz constant of $\phi^{\boldsymbol{v}}_{t}$.
  This means 
  \[
  \Big|\int_{\Omega}f_{\varepsilon_{k}}(t,x)dx \Big| \leq N^{\frac{N}{2}} Lip(\phi^{\boldsymbol{v}}_{t})^{N} \|f_{0}\|_{\mathscr{L}^{1}(\Omega)}.
  \]
  Hence, the Lebesgue's dominated convergence theorem yields
  \[
      \lim_{\varepsilon_{k}\rightarrow \infty}\int_{0}^{1}\int_{\Omega}f_{\varepsilon_{k}}(t,x)dxdt 
      = \int_{0}^{1}\int_{\Omega}f(t,x)dxdt.
  \]
The fact that $\mathscr{L}^{\infty}([0,1];BV(\Omega))$ is continuously embedded into $\mathscr{L}^{2}([0,1]; \mathscr{L}^{\frac{N}{N-1}}(\Omega))$ implies that there exists a subsequence $\{f_{\varepsilon_{k_{j}}}\}_{\varepsilon_{k_{j}}}\subset \{f_{\varepsilon_{k}}\}_{\varepsilon_{k}}$  converging  weakly to $f$ in $\mathscr{L}^{2}([0,1];\mathscr{L}^{\frac{N}{N-1}}(\Omega))$. Hence, we obtain 
  \begin{equation}\label{3.6}
  \begin{aligned}
      \int_{0}^{1} &\int_{\Omega} f_{\varepsilon_{k_{j}}}\left[\partial_{t} \varphi+\varphi \operatorname{div} \boldsymbol{v}+\boldsymbol{v} \cdot \nabla \varphi\right] d x d t \\[2mm]
      &\longrightarrow \int_{0}^{1} \int_{\Omega} f\left[\partial_{t} \varphi+\varphi \operatorname{div} \boldsymbol{v}+\boldsymbol{v} \cdot \nabla \varphi\right] d x d t 
  \end{aligned}
  \end{equation}
  for $\varphi \in \mathscr{C}_{c}^{\infty}([0,1[\times\Omega)$
  as $\varepsilon_{k_{j}} \rightarrow 0$.

  In addition, $f_{\varepsilon_{k_{j}}}$ is also the weak solution of the regularized transport equation  \eqref{regularizde transport}, which yields 
  \begin{equation}\label{3.7}
      \int_{0}^{1} \int_{\Omega} f_{\varepsilon_{k_{j}}}\left[\partial_{t} \varphi+\varphi \operatorname{div} \boldsymbol{v}+\boldsymbol{v} \cdot \nabla \varphi\right] d x d t = -\int_{\Omega} f_{0} \ast \rho_{\varepsilon_{k_{j}}}(x) \varphi(0, x) d x.
  \end{equation}
  Due to the property of approximate identity, one has
  \begin{equation}\label{3.8}
      -\int_{\Omega} f_{0} \ast \rho_{\varepsilon_{k_{j}}}(x) \varphi(0, x) d x \longrightarrow -\int_{\Omega} f_{0}(x) \varphi(0, x) d x\quad \text{as}~ \varepsilon_{k_{j}} \longrightarrow 0.
  \end{equation}
  Using \eqref{3.6}, \eqref{3.7} and \eqref{3.8}, we immediately deduce the following equality
  \[
      \int_{0}^{1} \int_{\Omega} f\left[\partial_{t} \varphi+\varphi \operatorname{div} \boldsymbol{v}+\boldsymbol{v} \cdot \nabla \varphi\right] d x d t = -\int_{\Omega} f_{0}(x) \varphi(0, x) d x.
  \]
  Hence, $f$ is a weak solution of transport equation \eqref{transport}.
\end{proof}
\begin{remark}
  Regarding the uniqueness of the weak solution to transport equation \eqref{transport}, it can be false without $\mathrm{div} \boldsymbol{v} = 0$ (see \cite{ambrosio2008transport}).
\end{remark}

The following theorem demonstrates that \eqref{3.5} is a continuous function with respect to time $t$ in $\|\cdot\|_{\mathscr{L}^{\infty}(\Omega)}$.
\begin{theorem}\label{th3.5}
Let $f_{0} \in \mathscr{L}^{\infty}(\Omega)$ and $\rho_{\varepsilon}$ be the mollifier in \cref{define2.7}. Furthermore, $\phi^{\boldsymbol{v}}_{t}$ is the associated flow of velocity field $\boldsymbol{v} \in \mathscr{L}_{\mathscr{V}}^{2}(\Omega)$ in $\Omega$ for $t\in [0,1]$. Define
    \begin{equation}\label{3.5}
        f_{\varepsilon}(t,x) \triangleq (f_{0} \ast \rho_{\varepsilon} )\circ (\phi^{\boldsymbol{v}}_{t})^{-1}(x).
    \end{equation}
    Then, $f_{\varepsilon} \in \mathscr{C}([0,1];\mathscr{L}^{\infty}(\Omega))$.
\end{theorem}
\begin{proof} 
  we deduce that $\phi^{\boldsymbol{v}}(\cdot,x)^{-1}$ is absolutely continuous in $[0, 1]$ for $x\in \Omega$ according to Carath\'eodory theorem in \cite{ambrosiointroduzione}. Hence, it holds 
  \[
      \|\phi^{\boldsymbol{v}}(t, x)^{-1} - \phi^{\boldsymbol{v}}(s, x)^{-1} \|_{\mathbb{R}^{N}} \longrightarrow 0
  \]
  for $x\in \Omega$ as $t \longrightarrow s$. Furthermore, the property of mollifier yields
  \[
      f_{0} \ast \rho_{\varepsilon} \in \mathscr{C}(\Omega, \mathbb{R}),
  \]
  that is, 
  \[
      \|f_{0} \ast \rho_{\varepsilon}(y_{1}) - f_{0} \ast \rho_{\varepsilon}(y_{2})\|_{\mathbb{R}} \longrightarrow 0
  \]
  for $y_{1}, y_{2}\in \Omega$ as $y_{1} \longrightarrow y_{2}$.  Therefore, we obtain
  \[
      \begin{aligned}
          \| f_{\varepsilon}(t,x) - f_{\varepsilon}(s,x) \|_{\mathscr{L}^{\infty}(\Omega)}
          =\,&\mathrm{ess} \underset{x\in \Omega}{\sup} |(f_{0} \ast \rho_{\varepsilon} )\circ (\phi^{\boldsymbol{v}}_{t})^{-1}(x) - (f_{0} \ast \rho_{\varepsilon} )\circ (\phi^{\boldsymbol{v}}_{s})^{-1}(x)|\\[2mm]
          &\longrightarrow 0\quad \text{as}~ t \longrightarrow s.
      \end{aligned}
  \]
  This completes the proof. 
\end{proof}

In the following theorem, the similar result to \cref{transport solution} with initial value $f_{0} \in \mathscr{L}^{\infty}(\Omega) \cap BV(\Omega)$ is concluded.
\begin{theorem}[Existence of weak solution with modified initial condition]\label{modify tranport solution}
 Let $\boldsymbol{v} \in \mathscr{L}_{\mathscr{V}}^{2}(\Omega)$ with $\mathrm{div} \boldsymbol{v} \in \mathscr{L}^{N}([0,1] \times \Omega)$ and  $f_{0} \in \mathscr{L}^{\infty}(\Omega) \cap BV(\Omega)$. Furthermore, $\phi^{\boldsymbol{v}}_{t}$ is the associated flow of velocity field $\boldsymbol{v} \in \mathscr{L}_{\mathscr{V}}^{2}(\Omega)$ in $\Omega$ for $t\in [0,1]$. Define
  \[
      f(t,x) \triangleq f_{0} \circ (\phi^{\boldsymbol{v}}_{t})^{-1}(x).
  \]
  Then, $f\in \mathscr{L}^{\infty}([0,1] \times \Omega) \cap \mathscr{L}^{\infty}([0,1];BV(\Omega))$ and is a weak solution to  transport equation \eqref{transport}.
\end{theorem}
\begin{proof}
  Similarly to the proof of \cref{transport solution}. 
  First of all, let $f_{0} \in \mathscr{L}^{\infty}(\Omega)$. Then, transport equation with initial value $f_{0} \ast \rho_{\varepsilon}$, i.e.,
  \[
         \left\{\begin{array}{ll} \partial_t f(t, \cdot)+\langle\nabla f(t, \cdot), \boldsymbol{v}(t, \cdot)\rangle_{\mathbb{R}^N}=0 \quad \text{in}~ ]0,1] \times \Omega,\\[2mm]
         f(0, \cdot)=f_{0} \ast \rho_{\varepsilon} \qquad \text { in }~ \Omega, \end{array}\right.
 \]
 has the unique solution for $\varepsilon > 0$ as below.
 \begin{equation}\label{explicit formula}
     f_{\varepsilon}(t,x) = (f_{0} \ast \rho_{\varepsilon}) \circ (\phi^{\boldsymbol{v}}_{t})^{-1}(x).
 \end{equation}
 Given $t\in [0, 1]$, we immediately obtain from \eqref{explicit formula} that $\{f_{\varepsilon}(t,\cdot)\}_{\varepsilon}$ is uniformly bounded in $\mathscr{L}^{\infty}(\Omega)$. Furthermore, \cref{th3.5} yields that $\{f_{\varepsilon}\}_{\varepsilon}$ is uniformly bounded 
 in $\mathscr{L}^{\infty}([0,1] \times \Omega)$. Hence, there exists a subsequence $\{f_{\varepsilon_{k}}\}_{\varepsilon_{k}}\subset \{f_{\varepsilon}\}_{\varepsilon}$  converging weakly$^{\star}$ to $g \in \mathscr{L}^{\infty}([0,1] \times \Omega)$. This implies that for $\varphi \in \mathscr{C}_{c}^{\infty}([0,1[\times\Omega)$ it holds 
 \[
 \begin{aligned}
     \int_{0}^{1} &\int_{\Omega} f_{\varepsilon_{k}}\left[\partial_{t} \varphi+\varphi \operatorname{div} \boldsymbol{v}+\boldsymbol{v} \cdot \nabla \varphi\right] d x d t\\[2mm]
     &\longrightarrow \int_{0}^{1} \int_{\Omega} g \left[\partial_{t} \varphi+\varphi \operatorname{div} \boldsymbol{v}+\boldsymbol{v} \cdot \nabla \varphi\right] d x d t 
 \end{aligned}
 \]
 as $\varepsilon_{k} \rightarrow 0$.
 
 By $f_{0} \in BV(\Omega)$ and the explicit formula \eqref{explicit formula}, $\{f_{\varepsilon_{k}}\}_{\varepsilon_{k}}$ is also uniformly bounded in $\mathscr{L}^{\infty}([0,1];BV(\Omega))$. 
 Thus, according to the proof of \cref{transport solution}, we can find a subsequence $\{f_{\varepsilon_{k_{j}}}\}_{\varepsilon_{k_{j}}}\subset \{f_{\varepsilon_{k}}\}_{\varepsilon_{k}}$  converging weakly to $\Tilde{f}$ in $\mathscr{L}^{2}([0,1];\mathscr{L}^{\frac{N}{N-1}}(\Omega))$. Hence, we get 
 \[
 \begin{aligned}
     \int_{0}^{1} &\int_{\Omega} f_{\varepsilon_{k_{j}}}\left[\partial_{t} \varphi+\varphi \operatorname{div} \boldsymbol{v}+\boldsymbol{v} \cdot \nabla \varphi\right] d x d t\\[2mm]
     &\longrightarrow \int_{0}^{1} \int_{\Omega} \Tilde{f} \left[\partial_{t} \varphi+\varphi \operatorname{div} \boldsymbol{v}+\boldsymbol{v} \cdot \nabla \varphi\right] d x d t 
 \end{aligned}
 \]
 for $\varphi \in \mathscr{C}_{c}^{\infty}([0,1[\times\Omega)$
 as $\varepsilon_{k_{j}} \rightarrow 0$. Here, $\Tilde{f}(t,x) = f_{0} \circ (\phi^{\boldsymbol{v}}_{t})^{-1}(x)$ and $\Tilde{f}$ is a weak solution of transport equation \eqref{transport} with initial value $f_{0} \in \mathscr{L}^{\infty}(\Omega) \cap BV(\Omega)$. Hence, $g = \Tilde{f}$. 
 This completes the proof. 
\end{proof}

The uniqueness of weak solution to transport equation \eqref{transport} is quite a tricky problem. We now introduce the notion of renormalized solution. For more details, see \cite{diperna1989ordinary, ambrosio2004transport, crippa2008flow, jarde2018analysis}. 
\begin{definition}[Renormalized solution \cite{jarde2018analysis}]
  Let $\boldsymbol{v} \in \mathscr{L}_{\mathscr{V}}^{2}(\Omega)$ with $\mathrm{div} \boldsymbol{v} \in \mathscr{L}^{N}([0,1] \times \Omega)$ and $f_{0} \in \mathscr{L}^{\infty}(\Omega) \cap BV(\Omega)$. Furthermore, $f\in \mathscr{L}^{\infty}([0,1] \times \Omega) \cap \mathscr{L}^{\infty}([0,1];BV(\Omega))$ is a weak solution of transport equation \eqref{transport}. We call f a renormalized solution if for any $\mathscr{C}^{1}$-function $\beta :\mathbb{R} \rightarrow \mathbb{R}$, the function $\beta(f)$ is a weak solution of 
  \[
  \left\{\begin{array}{ll} \partial_t \beta(f(t, \cdot))+\langle\nabla \beta(f(t, \cdot)), \boldsymbol{v}(t, \cdot)\rangle_{\mathbb{R}^N}=0 \quad \text{in}~ ]0,1] \times \Omega,\\[2mm]
  \beta(f(0, \cdot))=\beta(f_{0}) \qquad \text { in }~ \Omega.\end{array}\right.
  \]
\end{definition}

In particular, we are mainly concerned about when the velocity field $\boldsymbol{v} \in \mathscr{L}_{\mathscr{V}}^{2}(\Omega)$ can make transport equation \eqref{transport} have only renormalized solutions, which is very important for the well-posedness of the solution of transport equation \eqref{transport}. This property is stated below.
\begin{definition}[Renormalization property of velocity field \cite{jarde2018analysis}]
    Let $\boldsymbol{v}$ be a velocity field in $\mathscr{L}_{\mathscr{V}}^{2}(\Omega)$ with $\mathrm{div} \boldsymbol{v} \in \mathscr{L}^{N}([0,1] \times \Omega)$. We say that $\boldsymbol{v}$ has the property of renormalization if, for every $f_{0} \in \mathscr{L}^{\infty}(\Omega) \cap BV(\Omega)$, every bounded solution of transport equation \cref{transport} with velocity field $\boldsymbol{v}$ and initial value $f_{0}$ is a renormalized solution.
\end{definition}

The following theorem demonstrates that some velocity fields have the property of renormalization in bounded spatial domain.
\begin{theorem}[Renormalization property \cite{jarde2018analysis}]
    Let $\boldsymbol{v}$ be a bounded velocity field belonging to $\mathscr{L}^{1}((0,1),BV_{0}(\Omega;\mathbb{R}^{N}))$, such that $\operatorname{div}\boldsymbol{v} \in \mathscr{L}^{1}((0,1)\times\Omega)$. Then $\boldsymbol{v}$ has the renormalization property.
\end{theorem}   

Hence, the following result can be immediately concluded.
\begin{corollary}\label{cor3.7.1}
     Let $\boldsymbol{v} \in \mathscr{L}_{\mathscr{V}}^{2}(\Omega)$ be a bounded velocity field, then $\boldsymbol{v}$ has the renormalization property.
\end{corollary}

\begin{theorem}[Uniqueness of weak solution]\label{th3.8.1}
 Let $\boldsymbol{v} \in \mathscr{L}_{\mathscr{V}}^{2}(\Omega)$ with $\mathrm{div} \boldsymbol{v} \in \mathscr{L}^{N}([0,1] \times \Omega)$ be a bounded velocity field and  $f_{0} \in \mathscr{L}^{\infty}(\Omega) \cap BV(\Omega)$. Furthermore, $\phi^{\boldsymbol{v}}_{t}$ is the associated flow of velocity field $\boldsymbol{v} \in \mathscr{L}_{\mathscr{V}}^{2}(\Omega)$ in $\Omega$ for $t\in [0,1]$. Define
  \[
      f(t,x) \triangleq f_{0} \circ (\phi^{\boldsymbol{v}}_{t})^{-1}(x).
  \]
  Then, $f\in \mathscr{L}^{\infty}([0,1] \times \Omega) \cap \mathscr{L}^{\infty}([0,1];BV(\Omega))$ and is the unique renormalized solution to transport equation \eqref{transport}.
\end{theorem}
\begin{proof}
  Since $\boldsymbol{v}$ has the renormalization property and $f_{0} \in \mathscr{L}^{\infty}(\Omega)$, the solution to transport equation \eqref{transport} then is unique (see theorem 3.1.9 in \cite{jarde2018analysis}).
  Furthermore, \cref{modify tranport solution} imples that $f(t,x) = f_{0} \circ (\phi^{\boldsymbol{v}}_{t})^{-1}(x)$ is the unique  renormalized solution. 
\end{proof}

\subsection{Convergence of the constraint}
This section focus mainly on the convergence of the constraint in the sense of distributions. We will start with some lemmas, which resemble lemma 2 and lemma 3 in \cite{chen2011image}.
\begin{lemma}\label{lem3.9}
  Let $\rho_{\varepsilon}$ be the mollifier in \cref{define2.7}. If $\{\boldsymbol{v}_{n}\}_{n}$ and $\{f_{0}^{n}\}_{n}$ are the uniformly bounded in $\mathscr{L}_{\mathscr{V}}^{2}(\Omega)$ and $BV(\Omega)$, respectively. Moreover, $\{\phi^{\boldsymbol{v}_{n}}_{t}\}_{n}$ is the sequence of diffeomorphism flows in $\Omega$ for $t\in [0,1]$, where $\phi^{\boldsymbol{v}_{n}}_{t}$ is generated by the velocity field $\boldsymbol{v}_{n} \in \mathscr{L}_{\mathscr{V}}^{2}(\Omega)$. 
Then, $ \{(f_{0}^{n}\ast \rho_{\varepsilon}) \circ (\phi^{\boldsymbol{v}_{n}}_{t})^{-1}\}_{n}$ is uniformly bounded in $BV(\Omega)$ for $t \in [0,1]$.
\end{lemma}
\begin{proof}
  Fix $0< t \leq 1$ and $\varepsilon > 0$. Define
  \[
      \Omega_{\varepsilon} \triangleq \{ x \in \Omega : \operatorname{dist}(x,\partial \Omega) > \varepsilon \} 
  \]
  such that $\mathcal{L}^{N}(\Omega \backslash \Omega_{\varepsilon})$ is finite, where $\mathcal{L}^{N}$ denotes the Lebesgue measure.

  Let  $Lip(\phi^{\boldsymbol{v}_{n}}_{t}) = \|\nabla\phi^{\boldsymbol{v}_{n}}_{t}\|_{\mathscr{L}^{\infty}(\Omega)^{N \times N}}$. Using Gronwall's lemma in \cite{aubert2006mathematical} yields
  \[
      Lip(\phi^{\boldsymbol{v}_{n}}_{t}) \leq \exp\left(\int_{0}^{t}Lip(\boldsymbol{v}_{n}(s,\cdot))ds\right).
  \]
  Since $\{\boldsymbol{v}_{n}\}_{n}$ is uniformly bounded in $\mathscr{L}_{\mathscr{V}}^{2}(\Omega)$, we obtain
  \[
      Lip(\phi^{\boldsymbol{v}_{n}}_{t}) \leq L,
  \]
  where $L$ is a positive constant. By \cref{Hadamard}, we derive
  \[
      \|\det\nabla\phi^{\boldsymbol{v}_{n}}_{t}\|_{\mathscr{L}^{\infty}(\Omega)} \leq N^{\frac{N}{2}} Lip(\phi^{\boldsymbol{v}_{n}}_{t})^{N} \leq N^{\frac{N}{2}} L^{N}, 
  \]
  which implies that $\{\det\nabla\phi^{\boldsymbol{v}_{n}}_{t}\}_{n}$ is uniformly bounded in $\mathscr{L}^{\infty}(\Omega)$ for $t \in [0, 1]$. 

  Let $M$ and $C_{1}$ denote the upper bounds of $\{f_{0}^{n}\}_{n}$ in $BV(\Omega)$ and $\{\det\nabla \phi^{\boldsymbol{v}_{n}}_{t}\}_{n}$ in $\mathscr{L}^{\infty}(\Omega)$, respectively.
  We now verify the $\mathscr{L}^{1}$-norm by setting $y = (\phi^{\boldsymbol{v}_{n}}_{t})^{-1}(x)$
  \[
      \begin{aligned}
          \int_{\Omega}|(f_{0}^{n}\ast \rho_{\varepsilon}) \circ (\phi^{\boldsymbol{v}_{n}}_{t})^{-1}(x)|dx 
          &=\int_{\Omega}| f_{0}^{n}\ast \rho_{\varepsilon} (y)||\det\nabla \phi^{\boldsymbol{v}_{n}}_{t}(y)|dy\\[2mm]
          &\leq  C_{1} \int_{\Omega}|f_{0}^{n}\ast \rho_{\varepsilon} (y)|dy\\[2mm]
          &\leq C_{1} \|f_{0}^{n}\|_{\mathscr{L}^{1}(\Omega)}.
      \end{aligned}
  \]
  It remains to show the variation norm. Let $y = (\phi^{\boldsymbol{v}_{n}}_{t})^{-1}( x)$, we can derive
   \[
      \begin{aligned}
          \int_{\Omega}|\nabla((f_{0}^{n}\ast \rho_{\varepsilon}) \circ (\phi^{\boldsymbol{v}_{n}}_{t})^{-1}(x))|dx 
          &=\int_{\Omega}|\nabla(f_{0}^{n}\ast \rho_{\varepsilon} (y))||\det\nabla \phi^{\boldsymbol{v}_{n}}_{t}(y)|dy\\[2mm]
          &\leq  C_{1} \int_{\Omega}|\nabla(f_{0}^{n}\ast \rho_{\varepsilon} (y))|dy.
      \end{aligned}
  \]
  The estimate of $\int_{\Omega}|\nabla(f_{0}^{n}\ast \rho_{\varepsilon} (y))|dy$ is as follow. By the definition of $\rho_{\varepsilon}$, there exists $C_{2} > 0$ such that 
  \[
      \int_{\Omega \backslash \Omega_{\varepsilon}}|\nabla\rho_{\varepsilon} |dy < C_{2}.
  \]
  Hence, it holds 
  \[
  \begin{aligned}
      \int_{\Omega}|\nabla(f_{0}^{n}\ast \rho_{\varepsilon} (y))|dy 
      &= \int_{\Omega \backslash \Omega_{\varepsilon}}|\nabla(f_{0}^{n}\ast \rho_{\varepsilon} (y))|dy + \int_{\Omega_{\varepsilon}}|\nabla(f_{0}^{n}\ast \rho_{\varepsilon} (y))|dy \\[2mm]
      &=\int_{\Omega \backslash \Omega_{\varepsilon}} \Big| \int_{\Omega} f_{0}^{n}(z) \nabla_{y}(\rho_{\varepsilon} (y - z))dz \Big|dy + \int_{\Omega_{\varepsilon}}|\nabla(f_{0}^{n}\ast \rho_{\varepsilon} (y))|dy \\[2mm]
      &\leq \int_{\Omega} |f_{0}^{n}(z)| \int_{\Omega \backslash \Omega_{\varepsilon}}  |\nabla_{y}(\rho_{\varepsilon} (y - z))|dy dz + \int_{\Omega_{\varepsilon}}|\nabla(f_{0}^{n}\ast \rho_{\varepsilon} (y))|dy \\[2mm]
      &\leq C_{2} \|f_{0}^{n}\|_{\mathscr{L}^{1}(\Omega)} + \int_{\Omega_{\varepsilon}}|\nabla(f_{0}^{n}\ast \rho_{\varepsilon} (y))|dy.
  \end{aligned}
  \]
  With the aid of \cref{pro2.4} and \cref{th2.5}, we have
  \[
      \int_{\Omega_{\varepsilon}}|\nabla(f_{0}^{n}\ast \rho_{\varepsilon} (y))|dy \leq \int_{\Omega_{\varepsilon}}|Df_{0}^{n}\ast \rho_{\varepsilon} (y)|dy \leq |Df_{0}^{n}|(\Omega).
  \]
  From these facts, it follows that 
   \[
      \begin{aligned}
          \int_{\Omega}|\nabla((f_{0}^{n}\ast \rho_{\varepsilon}) \circ (\phi^{\boldsymbol{v}_{n}}_{t})^{-1}(x))|dx 
          &\leq C_{1} C_{2} \|f_{0}^{n}\|_{\mathscr{L}^{1}(\Omega)} + C_{1} |Df_{0}^{n}|(\Omega)\\[2mm]
          &\leq (C_{1} C_{2} + C_{1})M.
      \end{aligned}
  \]
  This completes the proof.
\end{proof}
\begin{lemma}\label{lem3.11}
  Let $\rho_{\varepsilon}$ be the mollifier in \cref{define2.7}. If $\{\boldsymbol{v}_{n}\}_{n}$ and $\{f_{0}^{n}\}_{n}$ are the uniformly bounded in $\mathscr{L}_{\mathscr{V}}^{2}(\Omega)$ and $BV(\Omega)$, respectively. Moreover, $\phi^{\boldsymbol{v}_{n}}_{t}$ is generated by the velocity field $\boldsymbol{v}_{n} \in \mathscr{L}_{\mathscr{V}}^{2}(\Omega)$. Define
  \[
      f_{\varepsilon}^{n}(t,x) \triangleq (f_{0}^{n}\ast \rho_{\varepsilon}) \circ (\phi^{\boldsymbol{v}_{n}}_{t})^{-1}(x).
  \]
  Then, there exists a subsequence $\{f_{\varepsilon}^{n_{k}}\}_{n_{k}}$ of $\{f_{\varepsilon}^{n}\}_{n}$ that converges to some limit $f_{\varepsilon}$ in $\mathscr{L}^{2}([0, 1]; \mathscr{L}^{p}(\Omega))$ with $1 \leq p < \frac{N}{N-1}$. Furthermore, $\{f_{\varepsilon}^{n_{k}}(t,\cdot)\}_{n_{k}}$ converges to $f_{\varepsilon}(t,\cdot)$ in $\mathscr{L}^{p}(\Omega)$ with $1 \leq p < \frac{N}{N-1}$ and weakly to $f_{\varepsilon}(t,\cdot)$ in $\mathscr{L}^{\frac{N}{N-1}}(\Omega)$ for $t\in [0, 1]$.
\end{lemma}
\begin{proof}
  Fix $0 < t \leq 1$, \cref{lem3.9} implies that $\{f_{\varepsilon}^{n}(t,\cdot)\}_{n}$ is uniformly bounded in $BV(\Omega)$  for $t \in [0,1]$. Then, there exists a subsequence $\{f_{\varepsilon}^{n_{k}}(t,\cdot)\}_{n_{k}}\subset \{f_{\varepsilon}^{n}(t,\cdot))\}_{n}$ converging to some limit $f_{\varepsilon}(t,\cdot)$ in $\mathscr{L}^{p}(\Omega)$ with $1 \leq p < \frac{N}{N-1}$ as $n_{k} \rightarrow \infty$ according to \cref{th2.7}.

  Moreover, according to theorem 2.1.1 in \cite{aubert2006mathematical} and \cref{lem3.9}, we observe that there exists a subsequence $\{f_{\varepsilon}^{n_{k_{j}}}(t,\cdot)\}_{n_{k_{j}}}\subset \{f_{\varepsilon}^{n_{k}}(t,\cdot)\}_{n_{k}}$  converging weakly to some limit $g_{\varepsilon}(t,\cdot)$ in $\mathscr{L}^{\frac{N}{N-1}}(\Omega)$ as $n_{k_{j}} \rightarrow \infty$. Obviously, $\{f_{\varepsilon}^{n_{k_{j}}}(t,\cdot)\}_{n_{k_{j}}}$ converges to some limit $f_{\varepsilon}(t,\cdot)$ in $\mathscr{L}^{p}(\Omega)$ with $1 \leq p < \frac{N}{N-1}$ as $n_{k_{j}} \rightarrow \infty$. For convenience, let $\{f_{\varepsilon}^{n_{k}}(t,\cdot)\}_{n_{k}}$ denote $\{f_{\varepsilon}^{n_{k_{j}}}(t,\cdot)\}_{n_{k_{j}}}$.

  We claim now that $f_{\varepsilon}(t,\cdot) = g_{\varepsilon}(t,\cdot)$. Indeed,  $\{f_{\varepsilon}^{n_{k}}(t,\cdot)\}_{n_{k}}$ converges weakly to some limit $g_{\varepsilon}(t,\cdot)$ in $\mathscr{L}^{\frac{N}{N-1}}(\Omega)$ as $n_{k} \rightarrow \infty$. Hence, it holds
  \[
      \int_{\Omega}f_{\varepsilon}^{n_{k}}(t,x)\varphi(x)dx \longrightarrow \int_{\Omega}g_{\varepsilon}(t,x)\varphi(x)dx\quad \text{as}~ n_{k} \longrightarrow \infty 
  \]
  for $\varphi \in \mathscr{L}^{N}(\Omega)$. On the other hand, $\{f_{\varepsilon}^{n_{k}}(t,\cdot)\}_{n_{k}}$ converges to some limit $f_{\varepsilon}(t,\cdot)$ in $\mathscr{L}^{p}(\Omega)$ with $1 \leq p < \frac{N}{N-1}$ as $n_{k} \rightarrow \infty$. Then, for $\psi \in \mathscr{L}^{q}(\Omega)$ such that $\frac{1}{p} + \frac{1}{q} = 1$, we deduce 
  \[
      \int_{\Omega}f_{\varepsilon}^{n_{k}}(t,x)\psi(x)dx \longrightarrow \int_{\Omega}f_{\varepsilon}(t,x)\psi(x)dx\quad \text{as}~ n_{k} \longrightarrow \infty.
  \]
  The fact $\mathscr{L}^{q}(\Omega) \subset \mathscr{L}^{N}(\Omega)$ implies that $f_{\varepsilon}(t,\cdot) = g_{\varepsilon}(t,\cdot)$.

  It remains to show that $\{f_{\varepsilon}^{n_{k}}\}_{n_{k}}$ converges to in $\mathscr{L}^{2}([0,1];\mathscr{L}^{p}(\Omega))$ with $1 \leq p < \frac{N}{N-1}$ to $f_{\varepsilon}$. 
  Indeed, it follows that
  \[
      \| f_{\varepsilon}^{n_{k}}(t,\cdot) - f_{\varepsilon}(t,\cdot)\|_{\mathscr{L}^{p}(\Omega)} \longrightarrow 0.
  \]
 From the boundedness of $\{f_{\varepsilon}^{n_{k}}(t,\cdot)\}_{n_{k}}$ in $BV(\Omega)$ and \cref{th2.7}, one has
  \[
  \begin{aligned}
      \| f_{\varepsilon}^{n_{k}}(t,\cdot) - f_{\varepsilon}(t,\cdot)\|_{\mathscr{L}^{p}(\Omega)}
      &\leq \| f_{\varepsilon}^{n_{k}}(t,\cdot)\|_{\mathscr{L}^{p}(\Omega)} + \|f_{\varepsilon}(t,\cdot)\|_{\mathscr{L}^{p}(\Omega)}\\[2mm]
      &\leq \mathcal{L}^{N}(\Omega)^{1 - \frac{(N- 1)p}{N}} \| f_{\varepsilon}^{n_{k}}(t,\cdot)\|^{p}_{\mathscr{L}^{\frac{N}{N-1}}(\Omega)} + \|f_{\varepsilon}(t,\cdot)\|_{\mathscr{L}^{p}(\Omega)}\\[2mm]
      &\leq h_{\varepsilon}(t) + \|f_{\varepsilon}(t,\cdot)\|_{\mathscr{L}^{p}(\Omega)},
  \end{aligned}
  \]
  where $\mathcal{L}^{N}$ is the Lebesgue measure and $h_{\varepsilon}(t)$ is an integrable function. 
  
  Therefore, the Lebesgue's dominated convergence theorem yields
  \[
  \begin{aligned}
      \lim_{n_{k} \rightarrow \infty} &\int_{0}^{1} \| f_{\varepsilon}^{n_{k}}(t,\cdot) - f_{\varepsilon}(t,\cdot)\|_{\mathscr{L}^{p}(\Omega)}^{2}dt \\[2mm]
      =& \int_{0}^{1}\lim_{n_{k} \rightarrow \infty} \| f_{\varepsilon}^{n_{k}}(t,\cdot) - f_{\varepsilon}(t,\cdot)\|_{\mathscr{L}^{p}(\Omega)}^{2}dt = 0,
  \end{aligned}
  \]
  where $1 \leq p < \frac{N}{N-1}$.
  This completes the proof.
\end{proof}

The following theorem demonstrates the stability of the nonlinear solution operator $S_{tran}$.
\begin{theorem}[Stability of $S_{tran}$]\label{th3.12}
    Assume that $1\leq p < \frac{N}{N-1}$ with $\frac{1}{p} + \frac{1}{q} = 1$. Let $f_{0} \in \mathscr{L}^{\infty}(\Omega) \cap BV(\Omega)$ and $\boldsymbol{v} \in \mathscr{L}_{\mathscr{V}}^{2}(\Omega)$ be a bounded velocity field. Furthermore, let $\{\boldsymbol{v}_{n}\}_{n} \subset \mathscr{L}_{\mathscr{V}}^{2}(\Omega)$ and $\{f_{0}^{n}\}_{n} \subset \mathscr{L}^{\infty}(\Omega) \cap BV(\Omega)$ be two sequences such that
    \begin{enumerate}[(1)]
        \item $\{\boldsymbol{v}_{n}\}_{n}$ is uniformly bounded in $\mathscr{L}_{\mathscr{V}}^{2}(\Omega)$ and converges weakly to $\boldsymbol{v}$ in $\mathscr{L}^{2}([0,1];\\ \mathscr{L}^{q}(\Omega;\mathbb{R}^{N}))$,
        \item $\{\operatorname{div}\boldsymbol{v}_{n}\}_{n}$ is uniformly bounded in $\mathscr{L}^{2}([0,1]; \mathscr{L}^{\infty}(\Omega))$ and converges weakly$^{\star}$ to $\operatorname{div}\boldsymbol{v}$ in $\mathscr{L}^{2}([0,1]; \mathscr{L}^{\infty}(\Omega))$,
        \item $\{f_{0}^{n}\}_{n}$ is uniformly bounded in $\mathscr{L}^{\infty}(\Omega) \cap BV(\Omega)$ and converges weakly to $f_{0}$ in $\mathscr{L}^{\frac{N}{N-1}}(\Omega)$.
    \end{enumerate}  
    Let $\{f^{n}\}_{n} \subset \mathscr{L}^{\infty}([0,1] \times \Omega) \cap \mathscr{L}^{\infty}([0,1];BV(\Omega))$ be the sequence of the solution of transport equation \eqref{transport} with velocity field $\boldsymbol{v}_{n} \in \mathscr{L}_{\mathscr{V}}^{2}(\Omega)$ and initial value $f_{0}^{n} \in \mathscr{L}^{\infty}(\Omega) \cap BV(\Omega)$ such that
     \begin{enumerate}[(4)]
        \item $\{f^{n}\}_{n}$ converges to $f \in \mathscr{L}^{2}([0,1];\mathscr{L}^{p}(\Omega))$ in $\mathscr{L}^{2}([0,1];\mathscr{L}^{1}(\Omega))$.
    \end{enumerate} 
    Then, $f = S_{tran}(f_{0},\boldsymbol{v})$, that is, $f \in \mathscr{L}^{\infty}([0,1] \times \Omega) \cap \mathscr{L}^{\infty}([0,1];BV(\Omega))$ is the solution of transport equation \eqref{transport} with velocity field $\boldsymbol{v}$ and initial value $f_{0}$.
\end{theorem}
\begin{proof}
  Let $\varphi \in \mathscr{C}_{c}^{\infty}([0,1[\times\Omega)$. From \cref{def3.1}, it holds that
  \[
     \int_{0}^{1} \int_{\Omega} f^{n}\left[\partial_{t} \varphi+\varphi \operatorname{div} \boldsymbol{v}_{n}+\boldsymbol{v}_{n} \cdot \nabla \varphi\right] d x d t + \int_{\Omega} f_{0}^{n}(x) \varphi(0, x) d x = 0.
 \]
 Hence, we derive
 \[
 \begin{aligned}
     \Big |  \int_{0}^{1} &\int_{\Omega} f\left[\partial_{t} \varphi+\varphi \operatorname{div} \boldsymbol{v}+\boldsymbol{v} \cdot \nabla \varphi\right] d x d t + \int_{\Omega} f_{0}(x) \varphi(0, x) d x\Big |\\[2mm]
     &= \Big | \int_{0}^{1} \int_{\Omega} f^{n}\left[\partial_{t} \varphi+\varphi \operatorname{div} \boldsymbol{v}_{n}+\boldsymbol{v}_{n} \cdot \nabla \varphi\right] d x d t + \int_{\Omega} f_{0}^{n}(x) \varphi(0, x) d x \\[2mm]
     &\hspace{5mm}- \Big( \int_{0}^{1} \int_{\Omega} f\left[\partial_{t} \varphi+\varphi \operatorname{div} \boldsymbol{v}+\boldsymbol{v} \cdot \nabla \varphi\right] d x d t + \int_{\Omega} f_{0}(x) \varphi(0, x) d x\Big) \Big |\\[2mm]
     &\leq \Big| \int_{0}^{1} \int_{\Omega} f^{n}\left[\partial_{t} \varphi+\boldsymbol{v}^{n} \cdot \nabla \varphi\right] d x d t + \int_{\Omega} f_{0}^{n}(x) \varphi(0, x) d x \\[2mm]
     &\hspace{5mm}- \Big(\int_{0}^{1} \int_{\Omega} f\left[\partial_{t} \varphi+\boldsymbol{v} \cdot \nabla \varphi\right] d x d t + \int_{\Omega} f_{0}(x) \varphi(0, x) d x \Big)\Big| \quad (\star)\\[2mm]
     &\hspace{5mm} + \Big | \int_{0}^{1} \int_{\Omega} \varphi (f^{n}\operatorname{div} \boldsymbol{v}^{n} - f\operatorname{div} \boldsymbol{v}) dx dt \Big | \quad (\star\star).
 \end{aligned}
 \]
 The estimate for $(\star)$ is as follow.
 \[
 \begin{aligned}
     (\star)&\leq \Big| \int_{0}^{1} \int_{\Omega} f^{n}\left[\partial_{t} \varphi+\boldsymbol{v}_{n} \cdot \nabla \varphi\right] d x d t - \int_{0}^{1} \int_{\Omega} f\left[\partial_{t} \varphi+\boldsymbol{v} \cdot \nabla \varphi\right] d x d t \Big|\\[2mm]
     &\hspace{5mm} + \Big| \int_{\Omega} f_{0}^{n}(x) \varphi(0, x) d x - \int_{\Omega} f_{0}(x) \varphi(0, x) d x \Big|\\[2mm]
     &\leq \Big| \int_{0}^{1} \int_{\Omega} \partial_{t} \varphi (f^{n} -  f) d x d t \Big|  + \Big| \int_{0}^{1} \int_{\Omega} \nabla \varphi \cdot (f^{n} \boldsymbol{v}_{n} - f \boldsymbol{v}) d x d t \Big|\\[2mm]
     &\hspace{5mm} + \Big| \int_{\Omega} f_{0}^{n}(x) \varphi(0, x) d x - \int_{\Omega} f_{0}(x) \varphi(0, x) d x \Big|. \\
 \end{aligned}
 \]
 The condition $(4)$ and condition $(3)$ yield
 \[
     \int_{0}^{1} \int_{\Omega} \partial_{t} \varphi (f^{n} -  f) d x d t \longrightarrow 0 \quad \text{as}~ n \longrightarrow \infty,
 \]
and
 \[
     \int_{\Omega} (f_{0}^{n}(x) - f_{0}(x)) \varphi(0, x) d x \longrightarrow 0\quad \text{as}~ n \longrightarrow \infty,
 \]
 respectively. Moreover, from the conditions $(1)$ and $(4)$, we obtain
 \[
 \begin{aligned}
     \Big|\int_{0}^{1} &\int_{\Omega} \nabla \varphi \cdot (f^{n} \boldsymbol{v}_{n} - f \boldsymbol{v}) d x d t\Big|\\[2mm]
     &=\Big|\int_{0}^{1} \int_{\Omega} \nabla \varphi \cdot \boldsymbol{v}_{n}(f^{n}  - f ) d x d t + \int_{0}^{1} \int_{\Omega} \nabla \varphi \cdot f (\boldsymbol{v}_{n} - \boldsymbol{v} ) d x d t\Big|\\[2mm]
     &\leq\int_{0}^{1} \int_{\Omega} |\nabla \varphi \cdot \boldsymbol{v}_{n}(f^{n}  - f ) |d x d t + \Big|\int_{0}^{1} \int_{\Omega} \nabla \varphi \cdot f (\boldsymbol{v}_{n} - \boldsymbol{v} ) d x d t\Big|\\[2mm]
     &\leq \|\nabla \varphi\|_{\mathscr{L}^{\infty}([0,1] \times \Omega)^{N\times N}}\cdot \|\boldsymbol{v}_{n}\|_{\mathscr{L}^{2}([0,1];\mathscr{L}^{\infty}(\Omega;\mathbb{R}^{N}))} \cdot \|f^{n}-f\|_{\mathscr{L}^{2}([0,1];\mathscr{L}^{1}(\Omega))}\\[2mm]
     &\hspace{5mm} + \Big|\int_{0}^{1} \int_{\Omega} \nabla \varphi \cdot f (\boldsymbol{v}_{n} - \boldsymbol{v} ) d x d t\Big| \longrightarrow 0\quad \text{as}~ n \longrightarrow \infty.
 \end{aligned}
 \]
 Hence, one has
 \[
     (\star) \longrightarrow 0\quad \text{as}~ n \longrightarrow \infty.
 \]
 The estimate for $(\star\star)$ with the conditions (2) and (4) goes as follow.
 \[
 \begin{aligned}
     \Big | \int_{0}^{1} &\int_{\Omega} \varphi (f^{n}\operatorname{div} \boldsymbol{v}_{n} - f\operatorname{div} \boldsymbol{v}) dx dt \Big |\\[2mm]
     &= \Big | \int_{0}^{1} \int_{\Omega} \varphi \operatorname{div} \boldsymbol{v}_{n}  (f^{n} - f)dx dt + \int_{0}^{1} \int_{\Omega} \varphi f(\operatorname{div} \boldsymbol{v}_{n} - \operatorname{div} \boldsymbol{v}) dx dt \Big | \\[2mm]
     &\leq \|\varphi\|_{\mathscr{L}^{\infty}([0,1]\times \Omega)} \|\operatorname{div} \boldsymbol{v}_{n}\|_{\mathscr{L}^{2}([0,1]; \mathscr{L}^{\infty}(\Omega))}\|f^{n}-f\|_{\mathscr{L}^{2}([0,1];\mathscr{L}^{1}(\Omega))} \\[2mm]
     &\hspace{5mm} + \Big|\int_{0}^{1} \int_{\Omega} \varphi f(\operatorname{div} \boldsymbol{v}_{n} - \operatorname{div} \boldsymbol{v}) dx dt \Big | \longrightarrow 0\quad \text{as}~ n\longrightarrow \infty.
 \end{aligned}
 \]
 From all these facts, we get that
 \[
     \int_{0}^{1} \int_{\Omega} f\left[\partial_{t} \varphi+\varphi \operatorname{div} \boldsymbol{v}+\boldsymbol{v} \cdot \nabla \varphi\right] d x d t + \int_{\Omega} f_{0}(x) \varphi(0, x) d x = 0.
 \]
 Furthermore, \cref{th3.8.1} yields that $f\in \mathscr{L}^{\infty}([0,1] \times \Omega) \cap \mathscr{L}^{\infty}([0,1];BV(\Omega))$. 
\end{proof}

To proceed, the convergence of the constraint can be concluded based on the above results.
\begin{theorem}[Convergence of the constraint]\label{th3.13.1}
    Assume that $1\leq p < \frac{N}{N-1}$ with $\frac{1}{p} + \frac{1}{q} = 1$. Let $\{\boldsymbol{v}_{n}\}_{n} \subset \mathscr{L}_{\mathscr{V}}^{2}(\Omega)$ and $\{f_{0}^{n}\}_{n} \subset \mathscr{L}^{\infty}(\Omega) \cap BV(\Omega)$ be two sequences with the following properties:
    \begin{enumerate}[(1)]
        \item $\{\boldsymbol{v}_{n}\}_{n}$ is uniformly bounded in $\mathscr{L}_{\mathscr{V}}^{2}(\Omega)$,
        \item $\{f_{0}^{n}\}_{n}$ is uniformly bounded in $\mathscr{L}^{\infty}(\Omega) \cap BV(\Omega)$.
    \end{enumerate}
    Then, there exist subsequence $\{\boldsymbol{v}_{n_{k}}\}_{n_{k}}\subset \{\boldsymbol{v}_{n}\}_{n}$ and $\{f_{0}^{n_{k}}\}_{n_{k}}\subset \{f_{0}^{n}\}_{n}$ converging weakly to $\boldsymbol{v} \in \mathscr{L}_{\mathscr{V}}^{2}(\Omega)$ in $\mathscr{L}^{2}([0,1];\mathscr{L}^{q}(\Omega;\mathbb{R}^{N}))$, where $\boldsymbol{v}$ is bounded in $\mathscr{L}_{\mathscr{V}}^{2}(\Omega)$, and $f_{0} \in \mathscr{L}^{\infty}(\Omega) \cap BV(\Omega)$ in $\mathscr{L}^{\frac{N}{N-1}}(\Omega)$, respectively, such that
    \[
        S_{tran}(f_{0}^{n_{k}}, \boldsymbol{v}_{n_{k}}) \longrightarrow S_{tran}(f_{0}, \boldsymbol{v}),
    \]
    in the sense of distributions. 
\end{theorem}
\begin{proof}
  The property $(1)$ implies that there exists a subsequence $\{\boldsymbol{v}_{n_{k}}\}_{n_{k}}\subset \{\boldsymbol{v}_{n}\}_{n}$  converging weakly to $\boldsymbol{v} \in \mathscr{L}_{\mathscr{V}}^{2}(\Omega)$ in $\mathscr{L}^{2}([0,1];\mathscr{L}^{q}(\Omega;\mathbb{R}^{N}))$. Moreover, $\boldsymbol{v}$ is bounded in $\mathscr{L}_{\mathscr{V}}^{2}(\Omega)$.

  Obviously, $\mathscr{V} \subset \mathscr{C}_{0}^{1}(\Omega;\mathbb{R}^{N})$ is continuously embedded into $W^{1,\infty}(\Omega;\mathbb{R}^{N})$. Then, we get that $\mathscr{L}_{\mathscr{V}}^{2}(\Omega)$ is continuously embedded into $\mathscr{L}^{2}([0,1]; W^{1,\infty} (\Omega;\mathbb{R}^{N}))$ by Lemma 2.1.23 in \cite{dirks2015variational}. Therefore, it follows that
  \[
      \|\boldsymbol{v}_{n}\|_{\mathscr{L}^{2}([0,1]; W^{1,\infty}(\Omega;\mathbb{R}^{N}))} \leq C \| \boldsymbol{v}_{n}\|_{\mathscr{L}_{\mathscr{V}}^{2}(\Omega)}, 
  \]
  where $C$ is a constant. 
  This means that $\{\operatorname{div}\boldsymbol{v}_{n_{k}}\}_{n_{k}}\subset \{\operatorname{div}\boldsymbol{v}_{n}\}_{n}$ is uniformly bounded in $\mathscr{L}^{2}([0,1]; \mathscr{L}^{\infty}(\Omega))$. Hence, there exists a subsequence $\{\operatorname{div}\boldsymbol{v}_{n_{k_{j}}}\}_{n_{k_{j}}}\subset \{\operatorname{div}\boldsymbol{v}_{n_{k}}\}_{n_{k}}$ converging weakly$^{\star}$ to $\operatorname{div}\boldsymbol{v}$ in $\mathscr{L}^{2}([0,1]; \mathscr{L}^{\infty}(\Omega))$. Indeed,  it holds
  \[
          \int_{0}^{1}\int_{\Omega} \boldsymbol{v}_{n_{k_{j}}} \cdot \nabla\varphi dx dt \longrightarrow \int_{0}^{1}\int_{\Omega} \boldsymbol{v} \cdot \nabla\varphi dx dt, 
  \]
 namely,
  \[
          \int_{0}^{1}\int_{\Omega} \varphi\operatorname{div}\boldsymbol{v}_{n_{k_{j}}}  dx dt \longrightarrow \int_{0}^{1}\int_{\Omega} \varphi\operatorname{div}\boldsymbol{v}  dx dt 
  \]
  for $\varphi \in \mathscr{C}_{c}^{\infty}([0,1]\times \Omega)$.
  For convenience, let $\{\operatorname{div}\boldsymbol{v}_{n_{k}}\}_{n_{k}}$ denote $\{\operatorname{div}\boldsymbol{v}_{n_{k_{j}}}\}_{n_{k_{j}}}$. 

  The property (2), along with \cref{th2.7}, implies that there exists a subsequence $\{f_{0}^{n_{k_{j}}}\}_{n_{k_{j}}}\subset \{f_{0}^{n_{k}}\}_{n_{k}}$ that converges weakly$^{\star}$ to $f_{0}\in BV(\Omega)$ in $BV(\Omega)$ and weakly to $\Tilde{f_{0}}\in \mathscr{L}^{\frac{N}{N-1}}(\Omega)$ in $\mathscr{L}^{\frac{N}{N-1}}(\Omega)$. In fact, $f_{0} = \Tilde{f_{0}}$ is as follows.
  For $\varphi \in \mathscr{L}^{\infty}(\Omega) \subset \mathscr{L}^{N}(\Omega)$, we derive
   \[
  \int_{\Omega}f_{0}^{n_{k_{j}}}(x)\varphi(x)dx \longrightarrow \int_{\Omega}f_{0}(x)\varphi(x)dx\quad \text{as}~ n_{k_{j}} \longrightarrow \infty,
  \]
  and
  \[
  \int_{\Omega}f_{0}^{n_{k_{j}}}(x)\varphi(x)dx \longrightarrow \int_{\Omega}\Tilde{f_{0}}(x)\varphi(x)dx\quad \text{as}~ n_{k_{j}} \longrightarrow \infty.
  \]
  Hence, $f_{0} = \Tilde{f_{0}}$.

  Moreover, $\{f_{0}^{n_{k_{j}}}\}_{n_{k_{j}}}$ is also uniformly bounded in $\mathscr{L}^{\infty}(\Omega)$. We immediately deduce that there exists a subsequence $\{f_{0}^{n_{k_{j'}}}\}_{n_{k_{j'}}}\subset \{f_{0}^{n_{k_{j}}}\}_{n_{k_{j}}}$ that converges weakly$^{\star}$ to $\Tilde{\Tilde{f_{0}}}\in \mathscr{L}^{\infty}(\Omega)$ in $\mathscr{L}^{\infty}(\Omega)$. Then, one has  
  \[
      \int_{\Omega}f_{0}^{n_{k_{j'}}}(x)\varphi(x)dx \longrightarrow \int_{\Omega}\Tilde{\Tilde{f_{0}}}(x)\varphi(x)dx\quad \text{as}~ n_{k_{j'}} \longrightarrow \infty 
  \]
  for $\varphi \in \mathscr{L}^{1}(\Omega)$.
  Obviously, it is also valid that 
  \[
  \int_{\Omega}f_{0}^{n_{k_{j'}}}(x)\varphi(x)dx \longrightarrow \int_{\Omega}f_{0}(x)\varphi(x)dx\quad \text{as}~ n_{k_{j'}} \longrightarrow \infty 
  \] 
  for $\varphi \in \mathscr{L}^{N}(\Omega) \subset \mathscr{L}^{1}(\Omega)$.
  Therefore, $f_{0} = \Tilde{f_{0}} = \Tilde{\Tilde{f_{0}}} \in \mathscr{L}^{\infty}(\Omega)\cap BV(\Omega)$.
  Next, let $\{f^{n_{k}}\}_{n_{k}}$ denote $\{f^{n_{k_{j'}}}\}_{n_{k_{j'}}}$.

   Let $\rho_{\varepsilon}$ be the mollifier in \cref{define2.7}. we claim now that $\{f_{0}^{n_{k}}\ast \rho_{\varepsilon}\}_{n_{k}}$ is uniformly bounded in $\mathscr{L}^{\infty}(\Omega)\cap BV(\Omega)$ and converges weakly to $f_{0}\ast \rho_{\varepsilon} \in \mathscr{L}^{\infty}(\Omega)\cap BV(\Omega)$ in $\mathscr{L}^{\frac{N}{N-1}}(\Omega)$ as $n_{k} \rightarrow \infty$. 
  To be specific, it is clear that $\{f_{0}^{n_{k}}\ast \rho_{\varepsilon}\}_{n_{k}}$ is uniformly bounded in $BV(\Omega)$ by the proof of \cref{lem3.9}.
  The following estimate holds by the (2.1) in \cite{ambrosio2000functions}
  \[
      \|f_{0}^{n_{k}}\ast \rho_{\varepsilon}\|_{\mathscr{L}^{\infty}(\Omega)} \leq \|f_{0}^{n_{k}}\|_{\mathscr{L}^{\infty}(\Omega)} \|\rho_{\varepsilon}\|_{\mathscr{L}^{1}(\Omega)} 
  \]
  for fixed $\varepsilon > 0$. Hence, $\{f_{0}^{n_{k}}\ast \rho_{\varepsilon}\}_{n_{k}}$ is uniformly bounded in $\mathscr{L}^{\infty}(\Omega)\cap BV(\Omega)$.
  
  Furthermore, for $\varphi \in \mathscr{L}^{N}(\Omega)$, we deduce that
  \[
      \lim_{n_{k}\rightarrow \infty} \varphi(x)f_{0}^{n_{k}}\ast \rho_{\varepsilon}(x)
     = \varphi(x)\int_{\Omega}f_{0}(y) \rho_{\varepsilon}(x-y)dy,
  \]
  and 
  \[
  \begin{aligned}
      \Big| \varphi(x)\int_{\Omega}f_{0}^{n_{k}}(y) \rho_{\varepsilon}(x-y)dy\Big|
      &\leq  |\varphi(x)|\int_{\Omega}|f_{0}^{n_{k}}(y)| |\rho_{\varepsilon}(x-y)|dy\\[2mm]
      &\leq  \|\rho_{\varepsilon}\|_{\mathscr{L}^{\infty}(\Omega)}|\varphi(x)|\int_{\Omega}|f_{0}^{n_{k}}(y)|dy\\[2mm]
      &\leq M\|\rho_{\varepsilon}\|_{\mathscr{L}^{\infty}(\Omega)}|\varphi(x)|,
  \end{aligned}
  \]
  where $M$ denotes the upper bound of $\{f_{0}^{n}\}_{n}$.
  Hence, the Lebesgue's dominated convergence theorem yields 
  \[
      \lim_{n_{k}\rightarrow \infty} \int_{\Omega}\varphi(x)f_{0}^{n_{k}}\ast \rho_{\varepsilon}(x) dx 
      = \int_{\Omega} \varphi(x)f_{0}\ast \rho_{\varepsilon}(x) dx.
  \]
  This indicates that the above assertion is correct.
  
  To proceed, define
  \[
      f_{\varepsilon}^{n_{k}}(t,x) \triangleq (f_{0}^{n_{k}}\ast \rho_{\varepsilon}) \circ (\phi^{\boldsymbol{v}_{n_{k}}}_{t})^{-1}(x),
  \]
  where $\phi^{\boldsymbol{v}_{n_{k}}}_{t}$ is generated by the velocity field $\boldsymbol{v}_{n_{k}} \in \mathscr{L}_{\mathscr{V}}^{2}(\Omega)$.
  \Cref{lem3.11} yields that there exists a subsequence $\{f_{\varepsilon}^{n_{k_{j}}}\}_{n_{k_{j}}}\subset \{f_{\varepsilon}^{n_{k}}\}_{n_{k}}$ converging to $f_{\varepsilon} \in \mathscr{L}^{2}([0,1];\mathscr{L}^{p}(\Omega))$ in $\mathscr{L}^{2}([0,1];\mathscr{L}^{p}(\Omega))$, to $f_{\varepsilon}(t,\cdot)$ in $\mathscr{L}^{p}(\Omega)$ with $1 \leq p < \frac{N}{N-1}$ and weakly to $f_{\varepsilon}(t,\cdot)$ in $\mathscr{L}^{\frac{N}{N-1}}(\Omega)$ for $t\in [0,1]$. Hence, we can derive for $\varphi \in \mathscr{L}^{N}(\Omega)$ it holds that 
  \[
  \int_{\Omega}f_{\varepsilon}^{n_{k_{j}}}(t,x)\varphi(x) dx \longrightarrow \int_{\Omega}f_{\varepsilon}(t,x)\varphi(x) dx\quad \text{as}~  n_{k_{j}} \longrightarrow \infty.
  \]
  On the other hand, \cref{th3.12} implies
  \[
      f_{\varepsilon}(t,x) = (f_{0}\ast \rho_{\varepsilon}) \circ (\phi^{\boldsymbol{v}}_{t})^{-1}(x).
  \]
  Furthermore, the approximate property of mollifier yields 
  \[
  \int_{\Omega}f_{\varepsilon}^{n_{k_{j}}}(t,x)\varphi(x) dx \longrightarrow \int_{\Omega}f^{n_{k_{j}}}(t,x)\varphi(x) dx\quad \text{as}~  \varepsilon \longrightarrow 0,
  \]
  and 
  \[
  \int_{\Omega}f_{\varepsilon}(t,x)\varphi(x) dx \longrightarrow \int_{\Omega}f(t,x)\varphi(x) dx\quad \text{as}~  \varepsilon \longrightarrow 0.
  \]
  From all these results, we immediately obtain that $\{f^{n_{k_{j}}}(t,\cdot)\}_{n_{k_{j}}}$ converges weakly to $f(t,\cdot) \in \mathscr{L}^{\infty}(\Omega)\cap BV(\Omega)$ in  $\mathscr{L}^{\frac{N}{N-1}}(\Omega)$ for $t\in [0,1]$ and  $f(t,x) = f_{0} \circ (\phi^{\boldsymbol{v}}_{t})^{-1}(x)$. Moreover, $f \in \mathscr{L}^{\infty}([0,1] \times \Omega) \cap \mathscr{L}^{\infty}([0,1];BV(\Omega))$ by \cref{th3.8.1}.
  
  Furthermore, it holds that 
  \[
      \int_{\Omega}|f_{\varepsilon}^{n_{k_{j}}}(t,x) - f_{\varepsilon}(t,x)| dx \longrightarrow 0\quad \text{as}~  n_{k_{j}} \longrightarrow \infty,
  \]
  in $\mathscr{L}^{p}(\Omega)$ with $1 \leq p < \frac{N}{N-1}$. The approximate property of mollifier implies 
  \[
      \int_{\Omega}|f_{\varepsilon}^{n_{k_{j}}}(t,x) - f^{n_{k_{j}}}(t,x)| dx \longrightarrow 0\quad \text{as}~  \varepsilon \longrightarrow 0,
  \]
  and 
  \[
      \int_{\Omega}|f_{\varepsilon}(t,x) - f(t,x)| dx \longrightarrow 0\quad \text{as}~  \varepsilon \longrightarrow 0.
  \]
  Therefore, we obtain 
  \[
  \begin{aligned}
      \int_{\Omega}|&f^{n_{k_{j}}}(t,x) - f(t,x)| dx \\[2mm]
      &\leq \int_{\Omega}|f^{n_{k_{j}}}(t,x) - f_{\varepsilon}^{n_{k_{j}}}(t,x)|dx +\int_{\Omega}| f_{\varepsilon}^{n_{k_{j}}}(t,x) - f_{\varepsilon}(t,x)|dx \\[2mm]
      &\hspace{5mm}+\int_{\Omega}| f_{\varepsilon}(t,x) - f(t,x)| dx \longrightarrow 0\quad \text{as}~  n_{k_{j}} \longrightarrow \infty \ \text{and}\  \varepsilon \longrightarrow 0.
  \end{aligned}
  \]
  This shows that $\{f^{n_{k_{j}}}(t,\cdot)\}_{n_{k_{j}}}$ converges to $f(t,\cdot) \in \mathscr{L}^{\infty}(\Omega) \cap BV(\Omega)$ in  $\mathscr{L}^{1}(\Omega)$ for $t\in [0,1]$.
  
  Subsequently, it remains valid that 
  \[
      \int_{0}^{1}\|f^{n_{k_{j}}}_{\varepsilon} - f_{\varepsilon}\|_{\mathscr{L}^{p}(\Omega)}^{2}dt \longrightarrow 0\quad \text{as}~  n_{k_{j}} \longrightarrow \infty,
  \]
  and
  \[
   \|\det\nabla\phi^{\boldsymbol{v}_{n_{k_{j}}}}_{t}\|_{\mathscr{L}^{\infty}(\Omega)} \leq N^{\frac{N}{2}} Lip(\phi^{\boldsymbol{v}_{n_{k_{j}}}}_{t})^{N} \leq N^{\frac{N}{2}} L(t)^{N},
  \]
  where $L(t)$ is a constant associated to $t$.
  
  Hence, we get that 
  \[
  \begin{aligned}
     \| &f_{\varepsilon}^{n_{k_{j}}} (t,\cdot) - f^{n_{k_{j}}}(t,\cdot)\|_{\mathscr{L}^{p}(\Omega)} \\[2mm]
      &\leq  \| f_{\varepsilon}^{n_{k_{j}}}(t,\cdot)\|_{\mathscr{L}^{p}(\Omega)} + \|f^{n_{k_{j}}}(t,\cdot)\|_{\mathscr{L}^{p}(\Omega)}\\[2mm]
      &\leq \mathcal{L}^{N}(\Omega)^{1 - \frac{(N- 1)p}{N}} \| f_{\varepsilon}^{n_{k_{j}}}(t,\cdot)\|_{\mathscr{L}^{\frac{N}{N-1}}(\Omega)} + \|f^{n_{k_{j}}}(t,\cdot)\|_{\mathscr{L}^{p}(\Omega)}\\[2mm]
      &\leq \mathcal{L}^{N}(\Omega)^{1 - \frac{(N- 1)p}{N}} \|\det\nabla\phi^{\boldsymbol{v}_{n_{k_{j}}}}_{t}\|_{\mathscr{L}^{\infty}(\Omega)} 
      \| f_{0}^{n_{k_{j}}} \ast \rho_{\varepsilon}\|_{\mathscr{L}^{\frac{N}{N-1}}(\Omega)} + \|f^{n_{k_{j}}}(t,\cdot)\|_{\mathscr{L}^{p}(\Omega)}\\[2mm]
      &\leq N^{\frac{N}{2}} L(t)^{N} \mathcal{L}^{N}(\Omega)^{1 - \frac{(N- 1)p}{N}}  \| f_{0}^{n_{k_{j}}}\|_{\mathscr{L}^{\frac{N}{N-1}}(\Omega)} + \|f^{n_{k_{j}}}(t,\cdot)\|_{\mathscr{L}^{p}(\Omega)},\\
  \end{aligned}
  \]
  and
  \[
  \begin{aligned}
      \| f_{\varepsilon}&(t,\cdot) - f(t,\cdot)\|_{\mathscr{L}^{p}(\Omega)} \\[2mm]
       &\leq  \| f_{\varepsilon}(t,\cdot)\|_{\mathscr{L}^{p}(\Omega)} + \|f(t,\cdot)\|_{\mathscr{L}^{p}(\Omega)}\\[2mm]
       &\leq\mathcal{L}^{N}(\Omega)^{1 - \frac{(N- 1)p}{N}} \| f_{\varepsilon}(t,\cdot)\|_{\mathscr{L}^{\frac{N}{N-1}}(\Omega)} + \|f(t,\cdot)\|_{\mathscr{L}^{p}(\Omega)}\\[2mm]
       &\leq \mathcal{L}^{N}(\Omega)^{1 - \frac{(N- 1)p}{N}} 
      \|\operatorname{det}\nabla\phi^{\boldsymbol{v}}_{t}\|_{\mathscr{L}^{\infty}(\Omega)}\|f_{0}\|_{\mathscr{L}^{\frac{N}{N-1}}(\Omega)} + \|f(t,\cdot)\|_{\mathscr{L}^{p}(\Omega)}.
  \end{aligned}
  \]
  The approximate property of mollifier and Lebesgue's dominated convergence theorem yield 
  \[
      \int_{0}^{1}\|f^{n_{k_{j}}}_{\varepsilon} - f^{n_{k_{j}}}\|_{\mathscr{L}^{p}(\Omega)}^{2}dt \longrightarrow 0\quad \text{as}~  \varepsilon \longrightarrow 0,
  \]
  and 
  \[
      \int_{0}^{1}\|f_{\varepsilon} - f\|_{\mathscr{L}^{p}(\Omega)}^{2}dt \longrightarrow 0 \quad \text{as}~  \varepsilon \longrightarrow 0.
  \]
  Hence, we derive
  \[
  \begin{aligned}
      \int_{0}^{1}&\|f^{n_{k_{j}}} - f\|_{\mathscr{L}^{p}(\Omega)}dt \\[2mm]
      &\leq \int_{0}^{1}\|f^{n_{k_{j}}} - f^{n_{k_{j}}}_{\varepsilon}\|_{\mathscr{L}^{p}(\Omega)}dt + \int_{0}^{1} \|f^{n_{k_{j}}}_{\varepsilon} - f_{\varepsilon}\|_{\mathscr{L}^{p}(\Omega)}dt\\[2mm]
      &\hspace{5mm} + \int_{0}^{1} \| f_{\varepsilon} - f\|_{\mathscr{L}^{p}(\Omega)}dt\\[2mm]
      &\leq \Big(\int_{0}^{1}\|f^{n_{k_{j}}} - f^{n_{k_{j}}}_{\varepsilon}\|^{2}_{\mathscr{L}^{p}(\Omega)}dt\Big)^{\frac{1}{2}} + \Big(\int_{0}^{1} \|f^{n_{k_{j}}}_{\varepsilon} - f_{\varepsilon}\|^{2}_{\mathscr{L}^{p}(\Omega)}dt\Big)^{\frac{1}{2}} \\[2mm]
      &\hspace{5mm}+ \Big(\int_{0}^{1} \| f_{\varepsilon} - f\|^{2}_{\mathscr{L}^{p}(\Omega)}dt\Big)^{\frac{1}{2}}
      \longrightarrow 0\quad \text{as}~  n_{k_{j}} \longrightarrow \infty \ \text{and}\  \varepsilon \longrightarrow 0.
  \end{aligned}
  \]
  This implies that $\{f^{n_{k_{j}}}\}_{n_{k_{j}}}$ converges  to $f \in \mathscr{L}^{\infty}([0,1] \times \Omega) \cap \mathscr{L}^{\infty}([0,1];BV(\Omega))$ in $\mathscr{L}^{1}([0,1];\mathscr{L}^{p}(\Omega))$ with $1 \leq p < \frac{N}{N-1}$.
  This completes the proof. 
\end{proof}
\begin{remark}
  \Cref{th3.13.1} shows that 
  \begin{enumerate}[(1)]
      \item the subsequence $\{f^{n_{k}}\}_{n_{k}}$ converges to $f \in \mathscr{L}^{\infty}([0,1] \times \Omega) \cap \mathscr{L}^{\infty}([0,1];BV(\Omega))$ in $\mathscr{L}^{1}([0,1];\mathscr{L}^{p}(\Omega))$ such that $f = S_{tran}(f_{0}, \boldsymbol{v})$, where $f^{n_{k}} = S_{tran}(f_{0}^{n_{k}}, \boldsymbol{v}_{n_{k}})$.
      \item the subsequence $\{f^{n_{k}}(t,\cdot)\}_{n_{k}}$ converges to $f(t,\cdot) \in \mathscr{L}^{\infty}(\Omega) \cap BV(\Omega)$ in $\mathscr{L}^{1}(\Omega)$ and weakly to $f(t,\cdot) \in \mathscr{L}^{\infty}(\Omega) \cap BV(\Omega)$ in $\mathscr{L}^{p}(\Omega)$ for $t\in [0,1]$, where $1 < p \leq \frac{N}{N-1}$.
  \end{enumerate}
\end{remark}
\section{Existence of a Minimizer}\label{sec4}
In this section, we investigate the existence of minimizers of problem \eqref{transport equation} using the results of the previous sections. We start with the following lemma.
\begin{lemma}\label{lem4.1}
    Suppose that the sequence $\{f^{n}\}_{n} \subset \mathscr{L}^{\infty}(\Omega)\cap BV(\Omega)$ converges to $f\in \mathscr{L}^{\infty}(\Omega)\cap BV(\Omega)$ in $\mathscr{L}^{p}(\Omega)$ with $1 \leq p < \frac{N}{N-1} \leq 2$ and is uniformly bounded in $\mathscr{L}^{p}(\Omega)$. If $\mathcal{T}:\mathscr{L}^{p}(\Omega) \supset \mathscr{L}^{\infty}(\Omega)\cap BV(\Omega) \rightarrow \mathscr{L}^{2}(\widetilde{\Omega})$ is a bounded and linear operator, then there exists a subsequence $\{f^{n_{k}}\}_{n_{k}}\subset \{f^{n}\}_{n}$ such that $\{\mathcal{T}f^{n_{k}}\}_{n_{k}}$ converges weakly to $\mathcal{T}f$ in $\mathscr{L}^{2}(\widetilde{\Omega})$.
\end{lemma}
\begin{proof}
  It holds
  \[
      \|\mathcal{T}f^{n}\|_{\mathscr{L}^{2}(\widetilde{\Omega})} \leq \|\mathcal{T}\| \|f^{n}\|_{\mathscr{L}^{p}(\Omega)}.
  \]
  Therefore, $\{\mathcal{T}f^{n}\}_{n}$ is uniformly bounded in $\mathscr{L}^{2}(\widetilde{\Omega})$. Hence, there exists a subsequence $\{\mathcal{T}f^{n_{k}}\}_{n_{k}}$ of $\{\mathcal{T}f^{n}\}_{n}$  converging weakly to $h \in \mathscr{L}^{2}(\widetilde{\Omega})$.

  We now claim $h = \mathcal{T}f$.
  Indeed, Let $\varphi = 1 \in \mathscr{L}^{2}(\widetilde{\Omega})$,  we then obtain 
  \[
  \begin{aligned}
      \Big|\int_{\widetilde{\Omega}}\varphi(\mathcal{T}f^{n_{k}} - \mathcal{T}f) dx \Big| &\leq \int_{\widetilde{\Omega}}|\mathcal{T}f^{n_{k}} - \mathcal{T}f| dx\\
      &\leq \mathcal{L}^{N}(\widetilde{\Omega})^{\frac{1}{2}} \|\mathcal{T}f^{n_{k}} - \mathcal{T}f\|_{\mathscr{L}^{2}(\widetilde{\Omega})}\\
      &\leq \mathcal{L}^{N}(\Tilde{\Omega})^{\frac{1}{2}} \| \mathcal{T} \| \|f^{n_{k}} - f\|_{\mathscr{L}^{p}(\Omega)}
      \longrightarrow 0\quad \text{as}~ n_{k_{j}} \longrightarrow \infty.
  \end{aligned}
  \]
  Since the uniqueness of the limit, we immediately get that $h = \mathcal{T}f$.
\end{proof}
\begin{remark}
  The conclusion in \cref{lem4.1} still holds when $\mathscr{L}^{2}(\widetilde{\Omega})$ is replaced by any self-reflexive Banach space.
\end{remark}

To proceed, the solution existence of problem \eqref{transport equation} can be proved as follows.
\begin{theorem}\label{th4.2}
    Assume that $\mathscr{V}$ is an admissible Hilbert space, $f_{0} \in  \mathscr{X} = BV(\Omega) \cap \mathscr{L}^{\infty}(\Omega)$, and $\mathcal{T}_{t_{i}}:\mathscr{L}^{p}(\Omega)\rightarrow \mathscr{L}^{2}(\widetilde{\Omega})$ is a linear continuous operator, where $1 \leq p < \frac{N}{N-1}$. Let further $\mathcal{D}(\cdot,\cdot)$ denote the squared $\mathscr{L}^{2}-$norm and $\mathcal{R}_{1}$ be the TV regularization. Moreover, there exists an index $i$ such that $\mathcal{T}_{t_{i}}.1 \neq 0$.  Then, the variational problem in \eqref{transport equation}
    has a solution $(f_{0}^{*}, \boldsymbol{v}^{*}) \in \mathscr{L}^{\infty}(\Omega)\cap BV(\Omega) \times \mathscr{L}_{\mathscr{V}}^{2}(\Omega)$.
\end{theorem}
\begin{proof}
  Select the index $i$ such that $\mathcal{T}_{t_{i}}.1 \neq 0$.
  Let $\{f_{0}^{n}\}_{n} \subset \mathscr{L}^{\infty}(\Omega)\cap BV(\Omega)$ and $\{\boldsymbol{v}_{n}\}_{n} \subset  \mathscr{L}_{\mathscr{V}}^{2}(\Omega)$ be a minimizing sequence for the objective functional in \cref{transport equation}. This implies that
  \begin{equation}\label{4.8}
      \int_{0}^{1}\|\boldsymbol{v}_{n}(\tau, \cdot)\|_{\mathscr{V}}^2 \mathrm{~d} \tau \leq C,
  \end{equation}
  and 
  \begin{equation}\label{4.9}
      |Df_{0}^{n}|(\Omega) \leq C,
  \end{equation}
  where $C$  denotes a universal positive constant that may differ from line to line.
  Obviously, \cref{4.8} shows that $\{\boldsymbol{v}_{n}\}_{n}$ is uniformly bounded in $\mathscr{L}_{\mathscr{V}}^{2}(\Omega)$. This means that $f^{n}_{t}(x) = f^{n}(t,x) = f_{0}^{n} \circ (\phi^{\boldsymbol{v}_{n}}_{t})^{-1}(x)$ via \cref{th3.8.1}. Therefore, one can derive
  \begin{equation}\label{4.10}
  \begin{aligned}
      \| \mathcal{T}_{t_{i}}(f^{n}_{t_{i}}) - g_{t_{i}}\|_{\mathscr{L}^{2}(\widetilde{\Omega})}^{2}
      =\| \mathcal{T}_{t_{i}}(f_{0}^{n} \circ (\phi^{\boldsymbol{v}_{n}}_{t_{i}})^{-1}) - g_{t_{i}}\|_{\mathscr{L}^{2}(\widetilde{\Omega})}^{2} \leq C.\\
  \end{aligned}
  \end{equation}
  Let $w^{n} = \frac{1}{|\Omega|}\int_{\Omega}f_{0}^{n}dx$ and $h^{n} = f_{0}^{n} - w^{n}$. Then $\int_{\Omega}h^{n} dx = 0$ and $Dh^{n} = Df_{0}^{n}$. Hence $|Dh^{n}(\Omega)|\leq C$.

  The generalized Poincar\'e--Writinger inequality yields 
  \begin{equation}\label{4.11}
      \|h^{n}\|_{\mathscr{L}^{p}(\Omega)} \leq k |Dh^{n}|(\Omega)\leq C,
  \end{equation}
  and
  \begin{equation}
      \|h^{n}\|_{\mathscr{L}^{\infty}(\Omega)} \leq k' |Dh^{n}|(\Omega)\leq C,
  \end{equation}
  where $1 \leq p < \frac{N}{N-1} \leq 2$. Here, $k$ and $k'$ are constants.

  Furthermore, we observe from \eqref{4.10} with $h^{n} = f_{0}^{n} - w^{n}$ that  
  \[
      \| \mathcal{T}_{t_{i}}(h^{n} \circ (\phi^{\boldsymbol{v}_{n}}_{t_{i}})^{-1}) + \mathcal{T}_{t_{i}}(w^{n}) - g_{t_{i}}\|_{\mathscr{L}^{2}(\widetilde{\Omega})}^{2} \leq C.
  \]
 Hence, we obtain
  \[
      \begin{aligned}
          \|\mathcal{T}_{t_{i}}&(w^{n})\|_{\mathscr{L}^{2}(\widetilde{\Omega})}\\[2mm]
          &= \|(\mathcal{T}_{t_{i}}(h^{n} \circ (\phi^{\boldsymbol{v}_{n}}_{t_{i}})^{-1}) + \mathcal{T}_{t_{i}}(w^{n}) - g_{t_{i}})- (\mathcal{T}_{t_{i}}(h^{n} \circ (\phi^{\boldsymbol{v}_{n}}_{t_{i}})^{-1}) - g_{t_{i}})\|_{\mathscr{L}^{2}(\widetilde{\Omega})}\\[2mm]
          &\leq C + \|\mathcal{T}_{t_{i}}\| \|h^{n} \circ (\phi^{\boldsymbol{v}_{n}}_{t_{i}})^{-1}\|_{\mathscr{L}^{p}(\Omega)} + \|g_{t_{i}}\|_{\mathscr{L}^{2}(\widetilde{\Omega})}.\\
      \end{aligned}
  \]
  The estimate for $\|h^{n} \circ (\phi^{\boldsymbol{v}_{n}}_{t_{i}})^{-1}\|_{\mathscr{L}^{p}(\Omega)}$ is as follow. First,
  the proof of \cref{th3.13.1} implies that
  \[
      \|\det\nabla\phi^{\boldsymbol{v}_{n}}_{t_{i}}\|_{\mathscr{L}^{\infty}(\Omega)} \leq C.
  \]
  Let $y = \phi^{\boldsymbol{v}_{n}}_{t_{i}}(x)$, with \eqref{4.11}, one has
  \[
      \begin{aligned}
          \int_{\Omega}|h^{n} \circ (\phi^{\boldsymbol{v}_{n}}_{t_{i}})^{-1}(x)|^{p}dx
          &=\int_{\Omega}|h^{n}(y)|^{p} |\det\nabla\phi^{\boldsymbol{v}_{n}}_{t_{i}}(y)|dx \leq  C. \\[2mm]
      \end{aligned}
  \]
  Therefore, we get 
  \[
          \|\mathcal{T}_{t_{i}}(w^{n})\|_{\mathscr{L}^{2}(\Tilde{\Omega})} \leq  C,
  \]
   that is,
  \[
      \|\mathcal{T}_{t_{i}}(w^{n})\|_{\mathscr{L}^{2}(\widetilde{\Omega})} = \Big|\frac{1}{|\Omega|}\int_{\Omega}f_{0}^{n}dx\Big|\|\mathcal{T}_{t_{i}}.1\|_{\mathscr{L}^{2}(\widetilde{\Omega})} \leq C.
  \]
 We immediately deduce that $\Big\{\Big|\int_{\Omega}f_{0}^{n}dx\Big|\Big\}_{n}$ is uniformly bounded.

  With \eqref{4.11}, it holds that
  \[
      \begin{aligned}
          \|f_{0}^{n}\|_{\mathscr{L}^{\infty}(\Omega)} = &\|h^{n} + \frac{1}{|\Omega|}\int_{\Omega}f_{0}^{n}dx \|_{\mathscr{L}^{\infty}(\Omega)}\\[2mm]
          \leq &\|h^{n}\|_{\mathscr{L}^{\infty}(\Omega)} + \Big|\int_{\Omega}f_{0}^{n}dx\Big| \\[2mm]
          \leq & C.
      \end{aligned}  
  \]
  From all these facts, $\{f_{0}^{n}\}_{n}$ is uniformly bounded in $\mathscr{L}^{\infty}(\Omega)$ and $\mathscr{L}^{1}(\Omega)$. 
  Together with \eqref{4.9}, $\{f_{0}^{n}\}_{n}$ is uniformly bounded in $BV(\Omega)$.
   Therefore, $\{f_{0}^{n}\}_{n}$ is uniformly bounded in $\mathscr{L}^{\infty}(\Omega) \cap BV(\Omega)$.

   Next, \cref{th3.13.1} shows that there exist subsequence $\{\boldsymbol{v}_{n_{k}}\}_{n_{k}}\subset \{\boldsymbol{v}_{n}\}_{n}$ and $\{f_{0}^{n_{k}}\}_{n_{k}}\subset \{f_{0}^{n}\}_{n}$ such that
   \[
       \boldsymbol{v}_{n_{k}} \xrightharpoonup[\ \ \ ]\quad \boldsymbol{v} \in  \mathscr{L}_{\mathscr{V}}^{2}(\Omega)\quad \text{in} ~\mathscr{L}_{\mathscr{V}}^{2}(\Omega),
   \]
   and 
   \[
       f_{0}^{n_{k}} \longrightarrow  f_{0} \in \mathscr{L}^{\infty}(\Omega) \cap BV(\Omega)\quad \text{in the weak}^{*}\  \text{topology of}\  BV(\Omega),
   \]
  and 
  \[
      f^{n_{k}}(t_{i},\cdot) \longrightarrow f(t_{i},\cdot) \in \mathscr{L}^{\infty}(\Omega) \cap BV(\Omega)\quad \text{in}~ \mathscr{L}^{1}(\Omega),
  \]
   and
   \[
      f^{n_{k}}(t_{i},\cdot) \xrightharpoonup[\ \ \ ]\quad f(t_{i},\cdot) \in \mathscr{L}^{\infty}(\Omega) \cap BV(\Omega)\quad \text{in}~ \mathscr{L}^{\frac{N}{N-1}}(\Omega).
  \]
  Here, $f(t_{i},x) = f_{0} \circ (\phi^{\boldsymbol{v}}_{t_{i}})^{-1}(x)$.
  Moreover, it is clear that for $i\in \{1,\cdots,T\}$
  \[
       \boldsymbol{v}_{n_{k}} \xrightharpoonup[\ \ \ ]\quad  \boldsymbol{v} \in  \mathscr{L}_{\mathscr{V}}^{2}(\Omega)\quad \text{in}~\mathscr{L}^{2}([0,t_{i}];\mathscr{V}).
   \]
   Furthermore, via \cref{lem4.1}, we immediately deduce that there exists a subsequence $\{\mathcal{T}_{t_{i}}(f^{n_{k_{j}}}_{t_{i}})\}_{n_{k_{j}}}$ of $\{\mathcal{T}_{t_{i}}(f^{n_{k}}_{t_{i}})\}_{n_{k}}$ converging weakly to $\mathcal{T}_{t_{i}}(f_{t_{i}})$ in $\mathscr{L}^{2}(\widetilde{\Omega})$ for $i\in \{1,\cdots,T\}$. Let $\{\mathcal{T}_{t_{i}}(f^{n_{k}}_{t_{i}})\}_{n_{k}}$ denote $\{\mathcal{T}_{t_{i}}(f^{n_{k_{j}}}_{t_{i}})\}_{n_{k_{j}}}$.

   From the weak semicontinuity property of the convex function of measures and the weak semicontinuity of the norm, we obtain that
   \[
       \| \mathcal{T}_{t_{i}}(f_{t_{i}}) - g_{t_{i}}\|_{\mathscr{L}^{2}(\widetilde{\Omega})}^{2} \leq \varliminf_{n_{k} \rightarrow \infty} \| \mathcal{T}_{t_{i}}(f^{n_{k}}_{t_{i}}) - g_{t_{i}}\|_{\mathscr{L}^{2}(\widetilde{\Omega})}^{2} 
   \]
   and 
   \[
       \int_0^{t_{i}}\|\boldsymbol{v}(\tau, \cdot)\|_{\mathscr{V}}^2 \mathrm{~d} \tau \leq \varliminf_{n_{k} \rightarrow \infty} \int_0^{t_{i}}\|\boldsymbol{v}_{n_{k}}(\tau, \cdot)\|_{\mathscr{V}}^2 \mathrm{~d} \tau 
   \]
   for $i\in \{1,\cdots,T\}$,
   and
   \[
       \int_{\Omega}|Df_{0}| \leq \varliminf_{n_{k} \rightarrow \infty} \int_{\Omega}|Df_{0}^{n_{k}}|.
   \]
   Summing up all these results, we get that
   \[
       \mathcal{J}(f_{0},\boldsymbol{v}) \leq \varliminf_{n \rightarrow \infty}\mathcal{J}(f_{0}^{n},\boldsymbol{v}^{n}),
   \]
   where $\mathcal{J}$ as defined in \eqref{transport equation}.
   This implies that the variational problem in \eqref{transport equation} has a solution $(f_{0}, \boldsymbol{v}) \in \mathscr{L}^{\infty}(\Omega)\cap BV(\Omega) \times \mathscr{L}_{\mathscr{V}}^{2}(\Omega)$.
\end{proof}

By \cref{the2.7}, the following result remains valid.
\begin{corollary}
    Assume that $\mathscr{V}$ is an admissible Hilbert space, $I \in  \mathscr{X} = BV(\Omega) \cap \mathscr{L}^{\infty}(\Omega)$, and $\mathcal{T}_{t_{i}}:\mathscr{L}^{p}(\Omega)\rightarrow \mathscr{L}^{2}(\widetilde{\Omega})$ is a linear continuous operator, where $1 \leq p < \frac{N}{N-1}$. Let further $\mathcal{D}(\cdot,\cdot)$ denote the squared $\mathscr{L}^{2}-$norm and $\mathcal{R}_{1}$ be the TV regularization. Moreover, there exists an index $i$ such that $\mathcal{T}_{t_{i}}.1 \neq 0$. Then, the variational problem in \eqref{2.10} has a solution $(I^{*}, \boldsymbol{v}^{*}) \in \mathscr{L}^{\infty}(\Omega)\cap BV(\Omega) \times \mathscr{L}_{\mathscr{V}}^{2}(\Omega)$.
\end{corollary}

\section{Conclusion}
To prove the solution existence of the time-discretized variational model with intensity, edge feature and topology preservations proposed in \cite{chen2019new}, we considered the problem from the perspective of PDE-constrained optimal control. It is worth noting that the unknown velocity field is restricted into the admissible Hilbert space, and the unknown template image (the initial value of the sequence image) is modeled in the space of bounded variation functions. We firstly proved  the solution existence and uniqueness of the constraint of transport equation in the equivalent optimal control model. Consequently, the stability of the associated nonlinear solution operator is concluded. In turn, the closure of that constraint is proved in the sense of distributions. Finally, the solution existence of the equivalent optimal control model is obtained immediately. As a result, the problem of the solution existence to the time-discretized intensity-preserving variational model proposed in \cite{chen2019new} is solved.



\renewcommand\refname{References}
\bibliographystyle{plain}
\bibliography{references}



\end{sloppypar}
\end{document}